\documentclass[letterpaper,11pt,reqno]{amsart}
\usepackage{amsmath}
\usepackage{amsthm}
\usepackage{amssymb,mathrsfs}
\usepackage{latexsym}
\usepackage{amsfonts}
\usepackage{cases}
\usepackage{graphicx}
\usepackage{stmaryrd}
\usepackage{mathtools}
\usepackage{color}
\usepackage{xcolor}
\usepackage[normalem]{ulem} 
\usepackage[margin=1in]{geometry}
\usepackage[utf8]{inputenc}
\usepackage[margin=1in]{geometry}
\usepackage{verbatim}
\usepackage{caption,subcaption}
\usepackage{algorithm}
\usepackage{algpseudocode}
\usepackage[normalem]{ulem}
\makeatletter
\renewcommand{\fnum@algorithm}{\fname@algorithm}
\makeatother
\usepackage[colorlinks,citecolor=blue,urlcolor=blue]{hyperref}
\usepackage{bm}
\usepackage{upgreek}
\usepackage{comment}
\usepackage{stmaryrd}
\usepackage{appendix}
\def \dt {{\Delta t}}

\def\L{{\rm Lip}}

\numberwithin{equation}{section}
\newtheorem{definition}{Definition}[section]
\newtheorem{remark}{Remark}[section]
\newtheorem{theorem}{Theorem}[section]
\newtheorem{proposition}[theorem]{Proposition}
\newtheorem{lemma}[theorem]{Lemma}
\newtheorem{corollary}[theorem]{Corollary}

\newcommand{\be}{\begin{equation}}
	\newcommand{\ee}{\end{equation}}
\newcommand{\bee}{\begin{equation*}}
	\newcommand{\eee}{\end{equation*}}
\newcommand{\bi}{\begin{itemize}}
	\newcommand{\ei}{\end{itemize}}
\DeclareMathOperator{\tr}{tr}
\DeclareMathOperator*{\argmax}{arg\,max}

\usepackage{algorithm}
\usepackage{algpseudocode}
\usepackage{soul}

\def \E{\mathbb{E}}

\def \N{\mathbb{N}}
\def \bbN{\mathbb{N}}

\def \R{\mathbb{R}}

\def \calF{{\mathcal F}}

\def \bbP{{\mathbb P}}
\def \bbE{{\mathbb E}}

\def \calS{\mathcal{S}}

\def \E{\mathbb{E}}

\def \N{\mathbb{N}}
\def \I{\mathbb{I}}

\def \R{\mathbb{R}}

\def \bbS{\mathbb{S}}
\def \bbR{\mathbb{R}}

\def\t{\widetilde}
\def\l{\left}
\def\r{\right}
\def\ll{\l\lVert}
\def\rl{\r\rVert}

\newcommand{\beq}{\begin{equation}}
\newcommand{\eeq}{\end{equation}}
\newcommand{\lb}{\label}
\def \mL{\mathcal{L}}
\def \Sig{\Sigma}
\def \nb{\nabla}
\def \pt{\partial}
\def \lammin{\lambda_{\min}}
\def \Linf{\infty} 
\def \lam{\lambda} 
\def \E{\mathbb{E}}

\def \mru{\mathrm{u}}

\mathtoolsset{showonlyrefs}

\begin{document}
\title[Discretization error]{Discretization error from regularized Reinforcement Learning to continuous-time stochastic control}

\author[H. Pham]{Huyên Pham\textsuperscript{1}} \thanks{\textsuperscript{1}Centre de Mathématiques Appliquées  (CMAP),  Ecole Polytechnique, France
  ({huyen.pham@polytechnique.edu}).}

\author[Y. P. Zhang]{Yuming Paul Zhang\textsuperscript{2}} \thanks{\textsuperscript{2} Department of Mathematics \& Statistics, Auburn University, USA. ({yzhangpaul@auburn.edu}).}

\author[Y. Zhu]{Yuhua Zhu\textsuperscript{3}} \thanks{\textsuperscript{3} Department of Statistics and Data Science, University of California, Los Angeles, USA. ({yuhuazhu@ucla.edu}).}



\date{}	

\begin{abstract}
This paper establishes a rigorous connection between regularized discrete-time reinforcement learning (RL) and continuous-time stochastic optimal control. Specifically, classical RL algorithms are typically solving a regularized discrete-time Bellman equation. We study the discretization error, namely, the gap between the optimal policy induced by the regularized discrete-time Bellman equation and the true optimal feedback control of the underlying continuous-time stochastic control problem. By deriving quantitative convergence rates for this gap, we provide a rigorous foundation for understanding the stability and implementation of exploratory RL policies in stochastic continuous-time environments.
\end{abstract}

\maketitle

\textit{Key words:} Hamilton--Jacobi--Bellman equations, stochastic control, reinforcement learning, convergence rate.

\medskip

\textit{AMS subject classifications:}
  35F21, 60J60, 68W40, 93E20.


\tableofcontents

\section{Introduction}

This paper studies the discounted infinite-horizon stochastic optimal control with dynamics governed by stochastic differential equations (SDEs) 
\cite{FS06}, \cite{YZ99}. This problem is fundamental for modeling systems that evolve continuously under persistent uncertainty and lack a natural terminal time. The infinite-horizon formulation captures long-run performance through stationary or ergodic objectives, yielding time-consistent policies, while SDEs provide a realistic description of continuous-time noise and fluctuations. This framework underpins a wide range of applications, including portfolio optimization in finance\cite{merton1975optimum}, stabilization and navigation in robotics \cite{stengel1994optimal}, and resource management in energy and capacity investment problems \cite{dixit1994investment}.

{The optimal control for stochastic control problems can be formulated over different classes of admissible controls, most notably open-loop and closed-loop (feedback) controls. An open-loop control is defined as a process adapted to the underlying filtration (e.g., the Brownian filtration), and therefore depend on the entire past realization of the driving noise. This formulation provides a mathematically convenient and general admissible class for analysis, but such controls are not directly implementable in practice, as the underlying noise is not observable.} 
In contrast, feedback controls only depend on the current state of the system and are directly implementable from observable system trajectories. {In fully observed Markov stochastic control problems under standard regularity conditions, it is well known that optimizing over open-loop controls yields the same value as restricting to feedback controls. Moreover, it can be computed by solving the associated Hamilton--Jacobi--Bellman (HJB) equation \cite{FS06,YZ99,pham2009continuous,bertsekas2012dynamic}. Consequently, feedback controls play a central role in stochastic control, as they are both implementable and, in classical settings, achieve optimal performance without loss relative to the broader class of adapted controls.}

In many real-world applications, the system dynamics are unknown due to modeling complexity, unobserved factors, or high-dimensional environments, and only discretely sampled trajectory data are available. Reinforcement learning (RL) therefore provides a natural data-driven framework for approximating the optimal control directly from observed data \cite{sutton2018reinforcement,watkins1992q}. However, learning from limited and noisy data can lead to instability and overfitting. Regularized RL, particularly with entropy regularization, alleviates these issues by encouraging exploration and smoothing the optimization landscape, resulting in more stable learning and improved sample efficiency \cite{todorov2009efficient,ziebart2008maximum,haarnoja2018soft, watkins1992q, kearns2002near}. This raises a natural question: how far is the optimal feedback control learned via RL from the true optimal policy of the underlying continuous-time stochastic control problem?

{Existing RL algorithms based on discretely sampled trajectory data} primarily 
focus on solving discrete-time Bellman equations and characterizing algorithmic and statistical errors relative to the corresponding discrete-time optimal policy \cite{jiang2026reinforcement, bertsekas1995neuro, tsitsiklis1999average, jaakkola1993convergence}. {There is a growing literature on continuous-time reinforcement learning. One may end up a different formulation if continuously observed trajectories are available \cite{WZZ20,jia2023q}. However, when only discrete-time observations are available, the continuous-time objective is typically approximated by a discrete-time cumulative sum, effectively reducing the problem to a discrete-time Markov decision process (see Remark \ref{rmk:ralation to cts-time RL}). The resulting discretization error is largely overlooked in the existing literature. Quantifying and understanding this discretization error is the central focus of this paper.}   To be more specific, this paper focuses on the discretization error, namely, the gap between the optimal policy induced by the regularized discrete-time Bellman equation and the true optimal feedback control in the continuous-time stochastic control problem.
We decompose this error into two components. The first captures the discrepancy between the optimal policy of the regularized discrete-time Bellman equation and that of its continuous-time counterpart, arising from time discretization of both the control and value function. The second captures the gap between the regularized and unregularized continuous-time Bellman equations, which reflects the effect of entropy regularization.

We emphasize that our comparison is conducted at the level of feedback policies, rather than value functions. However, directly measuring distances between policies is not meaningful, as policies that differ substantially pointwise may still induce similar performance. Instead, we compare policies through the value functions induced under the true system dynamics, namely the continuous-time SDE. Analyzing this performance-based discrepancy of feedback policies is a central technical challenge of the paper.

Our problem differs a lot from the open-loop problem. For example, the work of \cite{Kryejp,reisinger2019improved} concerns piecewise constant control which is given by adapted random processes taking constant values on small time intervals. The authors obtained the discretization error estimate for the value functions by employing Krylov’s method of shaking the coefficients. The proof relies on the regularity property and the stability analysis of solutions to the regularized PDEs.

In this work, our objective extends beyond deriving error estimates for the optimal value functions; we aim to establish a fine estimate for the feedback control. This presents a unique challenge: while the optimal piecewise constant control depends on the full history of the process, a feedback control is restricted to a deterministic function. The shaking coefficients method and the regularization argument fail to provide direct estimates for feedback controls. 

In our setting, these feedback controls could be significantly more sensitive to small perturbations.
From a technical standpoint, once the discretization 
with time step $h$ is applied at the level of the feedback control, the regularity of the optimal discrete value function $V_h$ may be limited, see Proposition \ref{prop:gradpi}. Specifically, we can only show that $V_h$ and the corresponding optimal feedback control $\pi_h^*$ are Lipschitz continuous with a constant of $O( h^{-1/2})$. Consequently, $\widetilde{r}$ (defined in \eqref{eq:expS}) lacks sufficient regularity, which implies the increment $\widetilde{r}(X_{t+h}) - \widetilde{r}(X_t)$ can possibly be as large as $O(1)$. 


To address these difficulties, we adopt a relaxed control formulation. Instead of tracking each trajectory of $X_t$, we estimate the differences of the density functions of $X_t$ and the discrete trajectories via the Fokker-Planck equation. Using the uniform $L^\infty$-bounds of the density of the stochastic control, the relaxed control formulation reduces the problem to estimating the density functions of trajectories for each fixed control value $u$. Such a procedure leads to a surprisingly outcome that we obtain a nearly linear error in $h$ when applying feedback control discretely with a time step of size $h$. This fast convergence, however, comes at the cost of a dependence on the exploration approximation parameter $\lambda$.

After establishing error estimates for the discrete application of feedback controls, we are able to compare policies through different value functions. Furthermore, these estimates yield a straightforward corollary for comparing the corresponding optimal value functions.

\medskip

\noindent{\bf Contributions.}
The main contribution of the paper are twofolds. First, we characterize the discretization error.
Given a fixed exploration parameter $\lambda>0$, the optimal randomized feedback policy and the optimal value from the regularized RL approximation (see \eqref{eq:regVdis}) is $\pi_h^{*,\lambda}$ and $V_h^\lambda$. We plug this RL policy $\pi_h^{*,\lambda}$ into the continuous time RL problem and compare it with the continuous optimal value function. We also take the optimal randomized policy $\pi^{*,\lambda}$ from the continuous time problem into the discrete time RL problem and compare it with the discrete optimal value function $V_h^\lambda$. Let us use  $V^\lambda[\pi]$ (resp. $V_h^\lambda[\pi]$) to denote the value function of the continuous time (resp. discrete time) problem associated to a randomized policy $\pi$. $V^\lambda[\pi]$ (resp. $V_h^\lambda[\pi]$)  measures the performance of policy  $\pi$ under the continuous-time dynamics (resp. discrete time). 
We obtain the following estimates (see Theorem \ref{T.3.4} and Corollary \ref{C.6.2} and Theorem \ref{T.6.3} for details).  
To the best of our knowledge, this is the first result providing quantitative estimates for the discretization error of feedback controls.

\begin{theorem}
Under certain conditions,  there exists $C=C(\lambda_0)$ such that for any $h\in (0,1)$ and $\lambda\in(0,\lambda_0)$,  
\begin{align} 
\big| V^\lambda[\pi_h^{*,\lambda}] - V^\lambda \big|_{_\infty} 
\; + \; \big| V_h^\lambda[\pi^{*,\lambda}] - V_h^\lambda \big|_{_\infty} \;+\;\big| V^\lambda - V_h^\lambda \big|_{_\infty}
& \leq \;   C\lambda^{-N}h|\ln h|, 
\end{align}
\[
\big|v(\cdot)-V[{\pi}_h^{*,\lambda} ](\cdot)\big|_{_\infty}\leq C\lambda |\ln \lambda|+C\lambda^{-N} h|\ln h|.
\]
\end{theorem}
Here $N$ denotes the dimension of the control space. 
The estimates say that the approximation of randomized policy contributes an error of $\lambda|\ln\lambda|$ (see \cite{TZZ}), while the discrete feedback control approximation contributes an error  of $\lambda^{-N} h|\ln h|$. 
When this upper bound is sharp, it provides an optimal schedule for choosing $\lambda$ in terms of $h$. That is, when $\lambda:=h^{1/(N+1)}$, the upper bound achieves its minimum, 
$
C\lambda |\ln \lambda|+C\lambda^{-N} h|\ln h|= \t{O} (h^{{1}/(N+1)})$.

The second contribution of this paper arises as a byproduct of the main theorem: we improve regularity bounds for entropy-regularized RL under SDE-induced transition kernels. It is well known in the RL literature that  $\|V^\lambda_h\|_\infty = O((1-\gamma)^{-1})$ \cite{puterman2014markov, bertsekas2012dynamic}, which diverges as the discounted factor $\gamma\to1$. Even under Lipschitz transition dynamics, the same unfavorable dependence persists for the gradient bound \(\|\nabla V_h^\lambda\|_\infty\) \cite{asadi2018lipschitz}.  When approximating a continuous-time control problem, one has $\gamma = e^{-\beta h}$, which leads to $\|V_h^\lambda\|_\infty, \|\nabla V_h^\lambda\|_\infty = O(h^{-1})$, and hence diverge as $h \to 0$. In contrast, we show that when the transition kernel is induced by an SDE, both the value function and its gradient admit uniform-in-$h$ bounds (See Lemma \ref{lemma:bd of v vx}):
$\|V_h^\lambda\|_\infty, \;\|\nabla V_h^\lambda\|_\infty \le C$. 
We further establish analogous uniform-in-$h$ bounds for the optimal policy $\pi^{*,\lam}_h$ from regularized RL (see Lemmas~\ref{L.2.4} and \ref{L.6.2}). On the other hand, we find that $\ll \nb \pi^{*,\lam}_h \rl_\infty$ scales as $O(h^{-1/2})$ (see Proposition~\ref{prop:gradpi}), which suggests that the feedback controls can be sensitive to perturbations. To the best of our knowledge, such regularity results have not been established in the existing RL literature.

\medskip

\noindent {\bf Related works.}
The development of discretization methods for continuous-time stochastic optimal control has evolved from classical numerical analysis toward modern reinforcement learning (RL) frameworks. In the classical setting, a primary focus has been the approximation of value functions for controlled diffusion processes using piecewise constant policies. 
Krylov \cite{Kryejp} and Reisinger et al. \cite{reisinger2019improved} established that such approximations yield an error that is at most of order $1/4$. This foundational work has been extended to more complex settings, such as the discretization of both state and control action spaces. For instance, the paper 
\cite{bayraktar2023approximate} analyzed approximate Q-learning for controlled diffusions, providing upper bounds for the approximation error of the optimal value function when the continuous problem is subject to piecewise constant control processes. These approaches often rely on the framework of viscosity solutions and the Hamilton-Jacobi-Bellman (HJB) equation, as detailed in the comprehensive works of \cite{FS06}, \cite{YZ99}, \cite{pham2009continuous}. 

Another established line of work approximates continuous diffusions by discrete-time Markov chains on finite state spaces.  The seminal work of Kushner and Dupuis \cite{kusdup01} provides a systematic methodology for this via weak convergence, ensuring that the discrete-model value functions converge to the continuous-time HJB solution. Alternatively, quantization algorithms in \cite{pagphapri04} offer a way to discretize the state space into optimal grids, which is particularly effective for multi-dimensional stochastic control problems. Similarly, recent research on the stabilization of hybrid systems (e.g., \cite{xu2024razumikhin}) explores how discrete-time observations can ensure almost sure exponential stability.
There is also a growing body of work that quantifies discretization error in continuous state spaces, though primarily in the policy evaluation setting (i.e., for a fixed policy). \cite{mou2024bellman} analyzes the discrepancy between the discrete-time and continuous-time Bellman equations at the level of the value function and propose higher-order Bellman equations to reduce this error. \cite{zhu2024phibe} introduces an alternative formulation tailored to the continuous-time structure and shows that, under suitable conditions, it achieves smaller discretization error for policy evaluation.


In recent years, the literature has shifted toward exploratory formulations of stochastic control to bridge the gap with reinforcement learning. Wang, Zariphopoulou, and Zhou \cite{WZZ20} introduced a stochastic control approach to RL that utilizes exploratory dynamics to capture learning under entropy regularization. This has led to further investigations into how sampling stochastic policies at discrete time points affects convergence. Szpruch et al. \cite{szpruch2024optimal} and Giegrich et al. \cite{giegrich2024convergence} have explored these limits, particularly in linear-quadratic (LQ) settings with Gaussian policies, showing how sampling algorithms converge toward continuous learning limits. Bender and Thuan \cite{bender2024grid} have analyzed the ``grid-sampling limit'' of SDEs, providing a rigorous look at how the dynamics behave when actions are sampled only on a discrete grid. Additionally, Jia, Ouyang, and Zhang \cite{JOZ} recently proved that as the sampling mesh size tends to zero, the controlled state process converges weakly to the dynamics with coefficients aggregated according to the stochastic policy. When the diffusion coefficients are uncontrolled, they obtained convergence rate between the dynamics. 
However, these studies primarily address time-dependent (open-loop) controls or the global behavior of the state process under sampling.

Our work departs from these established directions by focusing on the quantitative discretization error for feedback controls. The existing literature along this dimension remains limited. \cite{guo2023reinforcement} studies the convergence of learned feedback controls for linear--convex models with jumps with respect to the number of episodes, rather than discretization error in time.  
Later, \cite{zhu2024phibe} establishes an $O(h)$ discretization error for the approximated feedback policy under an alternative, unregularized framework distinct from the standard regularized MDP formulation we study in this paper. We provide a more detailed comparison with \cite{zhu2025optimal} in Remark~\ref{rmk: compare phibe}.

\medskip

\noindent {\bf Outline.}
Section 2 defines the mathematical foundations, including classical optimal control, relaxed controls, and discretely sampled processes. Section 3 introduces the optimal feedback relaxed controls within the framework of Markov Decision Processes (MDPs), while Section 4 provides the core technical contribution through detailed error estimates between the discrete and continuous models. Section 5 explores further regularity properties of the MDP to support the stability of the solution. In Section 6, we prove the convergence rate to the classical optimal control problem. 

\section{Preliminaries}

\subsection{The classical optimal control problem}
Let $(\Omega, \mathcal{F}, \mathbb{P}, \{\mathcal{F}_t\}_{t \ge 0})$ be a filtered probability space on which we have a $d$-dimensional $\mathcal{F}_t$-adapted Brownian motion $(B_t, \, t \ge 0)$.
Let $U\subseteq \bbR^N$ be a generic action/control space, and $\nu = (\nu_t, \, t \ge 0)$ be a  control which is an $\mathcal{F}_t$-adapted process taking values in $U$. We denote by $\mathcal{U}$ the set of such admissible 
open-loop controls $\nu = (\nu_t, \, t \ge 0)$.

The classical stochastic control problem is to control the state variable $X_t$ valued in  $\mathbb{R}^d$, whose dynamics is governed by the following stochastic differential equation (SDE):
\begin{equation}
\label{eq:classicalS}
dX^{\nu}_t = b(X^{\nu}_t, \nu_t) dt + \sigma(X^{\nu}_t) dB_t,
\end{equation}
where $b: \mathbb{R}^d \times U \to \mathbb{R}^d$ denotes the drift term, and $\sigma: \mathbb{R}^d  \to \mathbb{R}^{d \times d}$ is the covariance matrix of the state variable. Throughout the manuscript, we will only consider the case when $\sigma$ is uncontrolled, and we denote $\Sigma$ $=$ $\sigma\sigma^T$. 
The superscript `$\nu$' in $X^{\nu}_t$ emphasizes the dependence of the state variable on the control $\nu=(\nu_t)_t$.
A feedback control is a control process $\nu$ in the form $\nu_t$ $=$ $\mru(X_t^\nu)$ for some measurable function $\mru$ $:$ $\R^d$ $\rightarrow$ $U$, called feedback deterministic policy. 


The goal of the control problem is to maximize the total discounted reward over infinite horizon, leading to the (optimal) value function:
\begin{equation}
\label{eq:classicalV}
v(x) = \sup_{\nu \in \mathcal{U}} \mathbb{E}\left[ \int_0^{\infty} e^{-\beta t}r(X^{\nu}_t, \nu_t) dt \bigg| X^{\nu}_0 = x\right], \qquad x \in \R^d, 
\end{equation}
where $r: \mathbb{R}^d \times U \to \mathbb{R}$ is a reward function, and  $\beta > 0$ is the discount factor.  
By a standard dynamic programming argument, the value function $v$ in \eqref{eq:classicalV} solves the HJB equation:
\begin{equation}
\label{eq:HJBclassical}
-\beta v(x) + \sup_{u \in U} \left[r(x,u) + b(x,u) \cdot \nabla v(x) + \frac{1}{2}\tr(\Sigma(x) \nabla^2v(x))\right] = 0.
\end{equation}
Furthermore, the optimal control is in feedback form, i.e., it is represented as a deterministic mapping from the current state to the action/control space:
{$\nu^{*}_t = \mru^*(X^{*}_t)$} where 
the optimal feedback policy $\mru^*:\bbR^d\to U$ is given by the following formula (under certain conditions so that the following is well-defined):  
\[
\mru^*(x) \; \in \;   \argmax_{u \in U} \left[r(x,u) + b(x,u) \cdot \nabla v(x) + \frac{1}{2}\tr(\Sigma(x) \nabla^2v(x))\right].
\] 
We refer readers to \cite{FS06, YZ99, P09} for detailed accounts of the classical stochastic control theory.

\subsection{Relaxed controls}

The relaxed control problem is concerned with a probability distribution of controls $\varpi = (\varpi_t(\cdot), \, t \ge 0)$ over the control space $ U$ from which each trial is sampled \cite{WZZ20}. 
For each $t$, $\varpi_t\in \mathcal{P}( U)$,  with $\mathcal{P}( U)$ being the set of absolutely continuous probability density functions on $ U$.
The exploratory state dynamics is
\begin{equation}
\label{eq:expS}
dX^\varpi_t = \widetilde{b}(X^\varpi_t, \varpi_t)dt +  {\sigma}(X^\varpi_t)dB_t,
\end{equation}
where the coefficient $\widetilde{b}(\cdot, \cdot)$ is defined by
\begin{equation}
\label{eq:tildebsig}
\widetilde{b}(x, \uppi):=\int_{ U} b(x,u) \uppi(u)du, \quad \text{for $(x,\uppi)\in \mathbb{R}^d\times \mathcal{P}( U)$.}
\end{equation}
The distributional control $\varpi = (\varpi_t, \, t \ge 0)$ is also known as the relaxed control,
and a classical control $\nu = (\nu_t, \, t \ge 0)$ is a degenerated relaxed control when $\varpi_t$ is taken as the Dirac mass at $\nu_t$.

For the optimization problem, to avoid the degenerate case and to  encourage exploration, Shannon's entropy is added to the objective function:
\begin{equation}
\label{eq:regV}
V(x) = \sup_{\varpi \text{ is admissible}} \mathbb{E}\bigg[ \int_0^{\infty} e^{-\beta t} \bigg( \t r(X^\varpi_t, \varpi_t)  - \lambda \int_{ U} \varpi_t(u) \ln \varpi_t(u) du  \bigg) dt \bigg| X^\varpi_0 = x \bigg],
\end{equation}
where, for any $\uppi\in\mathcal{P}( U)$,
\begin{equation}
\label{eq:tilder}
\widetilde{r}(x, \uppi):=\int_{ U} r(x,u) \uppi(u)du,
\end{equation}
and $\lambda > 0$ is a fixed weight parameter controlling the level of exploration, and the definition of $\varpi=(\varpi_t)$ being admissible is given by \cite[Definition 1]{TZZ}. 

Following \cite{WZZ20}, by dynamic programming, the HJB equation associated to  \eqref{eq:expS}--\eqref{eq:regV} is
\beq\label{eq:HJBreg}
\begin{aligned}
&-\beta {V} (x) + \sup_{\uppi \in \mathcal{P}( U)} \int_{ U} \left(r(x,u) + b(x,u) \cdot \nabla {V} (x) - \lambda \ln \uppi(u)\right) \uppi(u)du + \frac{1}{2}\tr(\Sigma(x) \nabla^2{V} (x)) = 0.
\end{aligned}
\eeq
Then, 
 the optimal distributional feedback policy  is obtained by solving the maximization problem in  \eqref{eq:HJBreg} with the constraints $\int_{ U} \uppi(u) du = 1$ and $\uppi(u) \ge 0$ a.e. on $ U$. Indeed, through the verification theorem argument and the fact that $V$ is $C^{2,\alpha}$ under certain conditions, the optimal control can be written as
 \beq\lb{2.1}
 \varpi^*_t(u)={\pi}^*(X^*_t,u),\qquad t \ge 0
 \eeq
where
\begin{equation}
\label{eq:regfeedbackC}
{\pi}^{*}(x,u) := \frac{\exp\left(\frac{1}{\lambda} \left[r(x,u) + b(x,u) \cdot \nabla {V} (x) \right] \right)}{\int_{ U} \exp\left(\frac{1}{\lambda} \left[r(x,u) + b(x,u) \cdot \nabla {V} (x) \right] \right) du}.
\end{equation}
Plugging this into \eqref{eq:HJBreg}  yields that $V$ satisfies the following nonlinear parabolic PDE:
\begin{equation}
\label{main}
\begin{array}{lcl}
-\beta V(x) + \lambda \ln \int_{ U} \exp\left[\frac{1}{\lambda} \big(r(x,u) + b(x,u) \cdot \nabla_x V(x) \big) \right] du+ \frac{1}{2}\tr(\Sigma(x)  \nabla_x^2V(x)) = 0. 
\end{array}
\end{equation}

As for general feedback randomized policy  ${\pi}:\bbR^d\to \mathcal{P}( U)$, we let $X^{{\pi}}_t$ satisfy \eqref{eq:expS} with $\varpi_t(u):={\pi}(X^\pi_t,u)$. The solution exists by the classical theory  as long as $\t b(x,\pi(x,\cdot))$ and $ \sigma(x)$ are Lipschitz continuous in $x$. For simplicity, let us take $\pi(x,u)$ to be uniformly Lipschitz continuous in $x$ and assume that $\int_ U 1du$ is finite. Then the SDE is well-posedness. Next, define
\beq\lb{2.2}
V[{\pi}](x) = \mathbb{E}\bigg[ \int_0^{\infty} e^{-\beta t} \bigg(  \t r(X^{{\pi}}_t, {\pi}(X_t^{{\pi}},\cdot) ) - \lambda \int_{ U} {\pi}(X_t^{{\pi}},u) \ln {\pi}(X_t^{{\pi}},u) du  \bigg) dt \bigg| X^{{\pi}}_0 = x \bigg].
\eeq
where $\t r$ is defined in \eqref{eq:tilder}.

\subsection{Discretely Sampled Processes}
The detailed framework of the discrete sampling is given in \cite{szpruch2024optimal,JOZ}.  For readers' convenience, we  briefly discuss the settings below.

We need to define a sequence of the state process with random actions sampled according to the feedback control ${\pi}$. For each $i\in\bbN_0$, let $(\Omega^i, \mathcal{F}^i, \mathbb{P}^i)$ be a probability space, $\xi^i$ a random variable, and $\varphi$ a function such that
\[
\Omega^{i} \ni \omega \mapsto \varphi(x,\xi^i(\omega))=u \in  U
\]
has distribution ${\pi}( x,u)$ under a measure $\mathbb{P}^{i}$. Moreover, they are mutually independent and are independent from $(\Omega,\calF,\bbP)$.
The discrete state dynamic is given by the following. For all $i\in\N_0$, we define iteratively for $t\in[ih,(i+1)h]$,
\begin{equation}\lb{2.3}
Y^{\pi}_t = Y_{ih}^{\pi} + \int_{ih}^t b(Y_s^{\pi},\nu_{ih})\,ds + \int_{ih}^t \sigma(Y_s^{\pi})\,dB_s,
\qquad \nu_{ih} = \varphi(Y_{ih}^{\pi},\xi^i) \sim \pi(Y_{ih}^\pi,\cdot)
\end{equation}
We will always use the letter `$Y$' to denote discretely sampled dynamics. The probability density function of $Y_h^\pi$ given $Y_0^\pi$ $=$ $x$ is given by $\int_{U}\rho_h(y|x,u)\pi(x,u)du$ where $\rho_t(y|x,u)=\rho(t,y)$ satisfies the Fokker-Planck equation: for $t\in [0,h]$,
\begin{equation}\label{def of rho_h}
    \partial_t \rho(t,y) = \nabla\cdot\l[-b(y,u)\rho(t,y) + \frac12\nabla\cdot\l(\sigma(y)\sigma(y)^\top\rho(t,y)\r)\r], \quad \rho(0,y) = \delta_x(y).
\end{equation}


Similarly to Remark 3.2 \cite{JOZ}, although this definition depends on the sampling scheme $\varphi$ and $(\xi^n)$, the error estimates of the sampled dynamics and the cost function using the sampled dynamics depend only on the properties of the control ${\pi}$.

Having the sampled dynamics, we consider the discrete cost function:
\begin{equation}
\label{eq:regVdis}
V_h[{\pi}](x) = \mathbb{E}\left[ \sum_{i=0}^\infty e^{-\beta ih }h \bigg({r}(Y_{ih}^{{\pi}}, \nu_{ih} )  - \lambda \int_{ U} {\pi}(Y_{ih}^{{\pi}},u)  \ln {\pi}(Y_{ih}^{{\pi}},u)  du  \bigg) \bigg| Y^{{\pi}}_0 = x \right].
\end{equation}

Next, we define
\beq\lb{2.16}
V_h(x):=\sup_{\pi} V_h[{\pi}](x).
\eeq
We will show in Proposition \ref{T.2.3} that there exists a feedback stochastic control $\pi^*_h $ that is the maximizer of the above MDP problem.

\begin{remark}\label{rmk:ralation to cts-time RL}
We emphasize that the discretization error between the feedback policy derived from the above Bellman equation \eqref{2.16} and the true continuous-time optimal control is also relevant to many continuous-time RL algorithms. 

For example, the algorithms proposed in Section 4.1 of \cite{jia2023q} are derived from the optimal control formulation \eqref{eq:expS} - \eqref{eq:regV}. These algorithms are applicable when continuous-time data are available as they require computing continuous-time integrals along trajectories of \( x_t \) under a fixed policy \( \pi\). When only discrete-time observations are available (as in Algorithms 1 and 2 of \cite{jia2023q}), integrals of the form \( \int f(x_t)\,dt \) are approximated by sums \( \sum_i f(x_{t_i})\,\Delta t \). This approximation implicitly corresponds to solving a discrete-time Bellman equation under a fixed policy, thereby introducing a discretization error. This error propagates through subsequent quantities. In particular, one obtains $q(s,a) = \frac{Q_{\Delta t}(s,a) - V(s)}{\Delta t}$, and the induced policy from $q(s,a)$: $\pi(a|s) \propto \exp\!\left(\frac{1}{\gamma} q(s,a)\right) \propto \exp\!\left(\frac{1}{\gamma\dt} Q_{\Delta t}(s,a)\right)$,
which coincides with the soft policy improvement scheme \cite{haarnoja2017reinforcement} for the discrete-time Bellman equation \eqref{2.16} with $\lambda = \gamma \dt$.

Similar phenomena arise broadly in classical continuous-time RL literature \cite{baird1994reinforcement, jia2022policy, doya2000reinforcement}: although methods are motivated by continuous-time formulations, their practical implementations operating on discrete-time data implicitly are solving discrete-time Bellman equation \eqref{2.16} (or their unregularized counterparts). However, the discretization error induced by this mismatch is typically overlooked. Addressing this gap is precisely the focus of this paper.
\end{remark}
\color{black}




\section{Optimal feedback relaxed controls of MDP}

In this section, we prove well-posedness of optimal feedback relaxed controls of MDP.
We will need the following assumptions: 
\begin{enumerate}

\item[(H1)] $ U\subseteq \R^N$ for some $N\geq 1$ is {compact}, and is equipped with the standard Euclidean metric, and $| U|<\infty$, where $|\cdot|$ denotes the volume.

\medskip
    
\item[(H2)] 
Suppose that $b(x,u)$ and $\Sigma(x)$ are uniformly $C^{2,\alpha}$ in $x$.  There exist $\lammin>0$ and $M_1,M_2\geq 1$ such that   $\Sigma(\cdot)\geq \lammin\I_d$, and
\[
\| b(\cdot,\cdot)\|_\infty,\,\, \|\sigma(\cdot)\|_\infty , \,\, \| b(\cdot,\cdot)\|_\L,\,\, \|\sigma(\cdot)\|_\L\leq M_1,
\]
\[
\|r(\cdot,\cdot)\|_\infty,\,\,\|r(\cdot,\cdot)\|_\L\leq M_2.
\]
\end{enumerate}

\vspace{1mm}
Let $\rho_h(y|x,u)$ be the solution to \eqref{def of rho_h} at $t = h$, which represents the p.d.f of $Y_h$ given $Y_0 = x$ for the stochastic process driven by  $dY_t = {b}(Y_t, u)dt + {\sigma}(Y_t)dB_t$. By dynamic programming principle, one has  
\begin{equation}\label{dis-be-soft}
   V_h(x) = \sup_{{\pi}(x,\cdot)}\l\{\int_{ U} \l[{r}_h^\lambda(x,u;\pi)  + \gamma \l(\int_{\R^d} V_h(y) \rho_h(y|x,u)dy \r) \r]{\pi}(x,u) du\r\}.
\end{equation}
where 
\[
\gamma:=e^{-\beta h}\qquad \text{and}
\qquad
{r}_h^\lambda(x,u;\pi) :=r(x,u)h-\lambda h\ln {\pi}(x,u)  .
\]
We provide the rigorous proof of the equivalency between the discrete-time optimal control problem \eqref{eq:regVdis}--\eqref{2.16} and the regularized discrete-time Bellman equation \eqref{dis-be-soft} in Appendix \ref{proof of r-be}.

Define the operator $T^*$ as follows,
\beq\lb{111}
T^* W(x) := \sup_{{\pi}(x,\cdot)}\E_{u\sim\pi(x,\cdot)}\l[ {r}_h^\lambda(x,u;\pi)   + \gamma \E_{Y^{{\pi}}_{h}\sim\rho_h(\cdot|x,u)}[W(Y^{{\pi}}_{h}) | Y_0 = x]\r] .
\eeq
We state key properties of the operator in the following lemmas, which will be used repeatedly later. 
\begin{lemma}[Argmax of $T^*$]\label{dis-operator}
Let $\rho_h$ be defined as the above and $\gamma=e^{-\beta h}$. We have
    $$ T^*W(x) =  \lambda h \ln \l[\int_{U} \exp\l(\frac{r(x,u)h + \gamma\int_{\R^d} W(y) \rho_h(y|x,u) dy}{\lambda h}\r) du\r]$$
    and the supremum of \eqref{111} is achieved at 
    \begin{equation}\label{dis-opt-control-form}
        {\pi}^W_h(u|x): = \frac{1}{Z(x)}\exp\l[\frac{r(x,u)h + \gamma\int  W(y) \rho_h(y|x,u) dy}{\lambda h}\r], 
    \end{equation}
\[
\text{with the normalization constant }\, Z(x) := \int_{U}\exp \left[\frac{r(x,u)h + \gamma \int W(y) \rho_h(y|x,u) dy}{\lambda h}\right]du.
\] 
\end{lemma}

From the above lemma, if given the optimal value function $V_h$ from \eqref{2.16}, one has the explicit form of the optimal feedback control of the regularized MDP problem, which is 
$\pi_h^*$ $=$  $\pi_h^{V_h}$ and
\begin{equation}\label{dis-opt-control}
    \pi^*_h(x,u)   =  \frac{1}{Z(x)}\exp\l(\frac{r(x,u)h + \gamma\int  V_h(y) \rho_h(y|x,u) dy}{\lambda h}\r), 
\end{equation}
where $Z(x)= \int_U\exp\l(\frac{r(x,u)h + \int V_h(y) \rho_h(y|x,u) dy}{\lambda}\r)du$  is a normalization constant.

\medskip

Next, we give the contractive property of the operator $T^*$ in the following Lemma. 
\begin{lemma}[Contraction]\label{dis-opt-contraction}
For any bounded functions $W_1$ and $W_2$,
    $$\ll  T^*W_1(\cdot) - T^*W_2(\cdot) \rl_\infty \leq \gamma \ll W_1(\cdot) - W_2(\cdot) \rl_\infty.$$
\end{lemma}

Based on the above two lemmas, we conclude with the following result. 
\begin{proposition}[Well-posedness of MDP feedback controls]\lb{T.2.3}
We have that \eqref{2.16} admits a unique solution $V_h(x)\in L^\infty(\R^d)$ that is the fixed point of $T^*$, and the supremum in \eqref{dis-be-soft} is achieved at $ \pi^*_h$ given by \eqref{dis-opt-control}.
\end{proposition}

Closely related results have appeared in the RL literature \cite{ziebart2008maximum, haarnoja2018soft}. We present the proofs in Appendix \ref{proof of lemma dis-operator} for completeness.

\medskip

In the following lemma, we show that $ \pi^*_h$ is uniformly bounded depending on $\lambda$.
\begin{lemma}[Boundedness of $\pi_h^*$]\lb{L.2.4}
Under the assumptions of {\rm (H1)(H2)}, suppose that $C_d\|\nabla\sigma\|_\infty h^\frac12\leq \frac12$ for some dimensional constant $C_d$. We have
    \[
    \begin{aligned}
        \ll  \pi^*_h(x,u) \rl_\infty \leq \exp\l(\frac{2\|r\|_\infty+16\|b\|_\infty\|\nabla V_h\|_\infty}{\lambda}\r).
    \end{aligned}
    \]
\end{lemma}
We remark that, as shown in Lemma \ref{lemma:bd of v vx}, provided that $\beta \geq 1 + 2\|\nabla b\|_\infty + \frac{\|\nabla \Sigma\|^2_\infty}{4\lambda_{\min}}$, we have the following bound:
\[
\ll\nb_xV_h(\cdot) \rl_\infty\leq   e^h \ll \nb_x r \rl_\infty.
\]

\begin{proof}
    By the formula of $\pi^*_h $ in \eqref{dis-opt-control}, 
    for any function $\phi(x)$ independent of $u$, we have
\[
    \pi^*_h(x,u)  =  \frac{1}{ Z_\phi(x)}\exp\l(\frac{r(x,u)h + \gamma\int_{\R^d}  V_h(y) \rho_h(y|x,u) dy-\gamma\phi(x)}{\lambda h}\r), 
\]
where $Z_\phi(x)= \int_{U}\exp\l(\frac{r(x,u)h + \gamma\int_{\R^d} V_h(y) \rho_h(y|x,u) dy-\gamma\phi(x)}{\lambda}\r)du$. 
Therefore,  to obtain an $L^\infty$-bound for $\pi^*_h$, it suffices to estimate 
\[
\int_{\R^d} V_h(y) \rho_h(y|x,u) dy -\phi(x)
\]
for a suitable choice of $\phi(x)$.

Let us fix $u_0\in U$, let $\rho_0(t,y)$ be the unique solution to \eqref{def of rho_h}, and $Y_{0,t}$ solves 
\[
dY_{0,t} = {b}(Y_{0,t}, u_0)dt + {\sigma}(Y_{0,t})dB_t\quad\text{and}\quad Y_{0,0} = x.
\]
Then direct computation yields
\beq\lb{4.12}
\begin{aligned}
|Y_{0,t}-Y_t| &\leq \int_0^t|b(Y_{0,s},u_0)-b(Y_{s},u)|ds+\int_0^t|\sigma(Y_{0,s})-\sigma(Y_{s})|dB_s \\
&\leq 2\|b\|_\infty t+\int_0^t|\sigma(Y_{0,s})-\sigma(Y_{s})|dB_s.
\end{aligned}
\eeq
By the Burkholder-Davis-Gundy inequality (see \cite[Chapter IV]{revuz2013continuous}), there exists a dimensional constant $C_d$ such that
\begin{align*}
\bbE\left[\sup_{t_0\in [0,{t}] }\left|\int_0^{t_0}  (\sigma(Y_{0,s} )-\sigma(Y_s))dB_s \right|\right]
&\leq C_d\,\bbE \left[\left(\int_0^{t} |\sigma(Y_{0,s} )-\sigma(Y_s)|^2 ds\right)^{1/2}\right]\\
&\leq C_d\|\nabla\sigma\|_\infty t^\frac12  \bbE\Big[\sup_{t_0\in [0,{t}]}|Y_{0,t_0}-Y_{0,t_0}| \Big].
\end{align*}
We get from \eqref{4.12} that
\begin{align*}
\bbE\sup_{t\in [0,h] }|Y_{0,t}-Y_t|
&\leq 2\|b\|_\infty h+C_d\|\nabla\sigma\|_\infty h^\frac12 \, \bbE\Big[\sup_{t\in [0,h]}|Y_{0,t}-Y_{t}| \Big].
\end{align*}
Consequently, if $C_d\|\nabla\sigma\|_\infty h^\frac12\leq \frac12$, 
\begin{align}\lb{4.10}
\bbE|Y_{0,h}-Y_h| \leq  4\|b\|_\infty h.
\end{align}

Now, recall that $\rho_h(y|x,u)$ is the p.d.f of $Y_h$ and so 
\[
\int_{\R^d} V_h(y) \rho_h(y|x,u) dy=\bbE V_h(Y_h).
\]
We also let $\rho_0$ be the p.d.f of $Y_{0,h}$ and then
\[
\phi(x):=\int_{\R^d} V_h(y) \rho_0(y) dy=\bbE V_h(Y_{0,h})
\]
which is independent of $u$. It follows from \eqref{4.10} that
\begin{align*}
   \left| \int_{\R^d} V_h(y) \rho_h(y|x,u) dy -\phi(x)\right|
&=\l|\bbE (V_h(Y_h)-V_h(Y_{0,h}))\r|\\
&\leq \|\nabla V_h\|_\infty\bbE|Y_{h}-Y_{0,h}|\leq 4\|b\|_\infty \|\nabla V_h\|_\infty h.
\end{align*}
With this selection of $\phi$,  we have
\begin{align*}
    \ll \pi^*_h \rl_\infty &\leq \frac{1}{|{U}|} \exp\l(\frac{2\ll r(x,u) \rl_\infty h+2\gamma(\int_{\R^d} V_h(y) \rho_h(y|x,u) dy -\phi(x))}{\lam h}\r)\\
    &\leq \exp\l(\frac{2\|r\|_\infty+8\|b\|_\infty\|\nabla V_h\|_\infty}{\lambda}\r).
\end{align*}
\end{proof}

The following results concern the regularity property of the continuous value function and are a consequence of  the classical PDE regularity theory.

\begin{lemma}[Regularity of $V$]\lb{L.3.5}
Under the assumptions of {\rm (H1)(H2)}, then $V$ from \eqref{eq:regV} is uniformly $C^{2,\alpha}$ for any $\alpha\in (0,1)$. Furthermore, 
there exists $C\geq 1$ depending only on $d,\lambda_0$ and the constants in {\rm (H2)}, but independent of $\beta\geq 1$ and any $\lambda\in (0,\lambda_0)$, such that
\[
\|V\|_\infty\leq \| r\|_\infty/\beta,\quad \|\nabla V\|_{\infty}\leq C/\sqrt{\beta}\quad\text{and}\quad \|D^2V\|_{\infty}\leq C.
\]
Consequently, ${\pi}^*(x,u)$ from \eqref{eq:regfeedbackC} is uniformly bounded and is Lipschitz continuous in $x$ with bounds of the form $\exp(C/\lambda)$.
\end{lemma}


\begin{proof}
Since $\pm \|\tilde r\|_\infty/\beta$ is a sub- and a super- solution, respectively, to \eqref{main}, the first claim follows from the comparison principle and that $\|\tilde r\|_\infty=\| r\|_\infty$.

Let $\calS^d$ be the space of $d\times d$ symmetric matrices equipped with the spectral norm.
Note that $V$ is the unique viscosity solution to
\[
F_\lambda(D^2V,\nabla V,V,x)=0
\]
where $F_\lambda: \mathcal{S}^d \times \mathbb{R}^d \times \mathbb{R} \times \mathbb{R}^d \to \mathbb{R}$ is given by
\[
F_\lambda(X,p,s,x):=\beta s - \lambda \ln \int_{ U} \exp\left[\frac{1}{\lambda} \big(r(x,u) + b(x,u) \cdot p \big) \right] du+ \frac{1}{2}\tr(\sigma(x) \sigma(x)^T X).
\]
It was verified in 
Theorem 10 \cite{TZZ} that the operator $F_\lambda$ satisfies the regularity assumptions (a)-(c) \cite{TZZ} uniformly for all $\lambda\in (0,\lambda_0)$, and it is concave in the sense of Definition 4 \cite{TZZ}. 
By direct computation, there exists a dimensional constant $C$ 
such that for any $x,y\in \bbR^d$ and $(X,p,q,s,t)\in \calS^d\times\bbR^{d}\times\bbR^{d}\times\bbR\times\bbR$, 
\begin{align*}
|F_\lambda(X,p,s,x)-F_\lambda(X,q,t,y)|&\leq \beta|s-t|+\|r\|_\L|x-y|+\|b\|_\infty|p-q|\\
&\quad +C(\|b\|_\L+\|\Sigma\|_\L)|x-y|(|p|+|q|+|X|),    
\end{align*}
\[
|F(0,0,0,x)|\leq \|r\|_\infty.
\]
If writing
\[
\beta_2(x,y)=C(\|b\|_\L+\|\Sigma\|_\L)|x-y|,\quad \gamma_2(x,y)=\|r\|_\L|x-y|,
\]
it follows from the Schauder estimate, see \cite[Theorem 5.1]{lian2020pointwise} (with the above $\beta_2$ and $\gamma_2$) and \cite{CC95}, that $V$ is uniformly $C^{2,\alpha}$ and
\beq\lb{4.44}
\|V\|_{C^{2,\alpha}}\leq C(\|V\|_\infty+\|r\|_{C^\alpha})\leq C(\|V\|_\infty+\|r\|_{\infty}+\|r\|_\L)
\eeq
for some $C$ possibly depending on $d,\beta,\lambda_0$ and the constants in (H2), but independent of $\lambda\in (0,\lambda_0)$.

To see that the $C^2$ norm is independent of $\beta\geq 1$,  
we apply a scaling argument.
Note that $w(x):={V}(x/\sqrt{\beta})$ solves
\[
G_{\lambda/\beta}(D^2w,Dw,w,x)=0
\]
where
\[
G_{\lambda/\beta}(X,p,s,x):= s - \frac\lambda\beta \ln \int_{ U} \exp\left[\frac{\beta}{\lambda} \big( r_\beta(x,u) +  b_\beta(x,u) \cdot p \big) \right] du+ \frac{1}{2}\tr(\sigma_\beta(x) \sigma_\beta(x)^T X)
\]
and
\[
r_\beta(x,u):= \frac{{r}(x/\sqrt{\beta},u)}\beta ,\qquad b_\beta(x,u):=\frac{{b}(x/\sqrt{\beta},u)}{\sqrt{\beta}} \quad\text{and}\quad \sigma_\beta(x,u):=\sigma(x/\sqrt{\beta}).
\]
Since $|V|\leq \| r\|_\infty/\beta$, then $|w|\leq \| r\|_\infty/\beta$.
Furthermore, it is direct to see that the $L^\infty$ and Lipchitz norms of 
\[
\beta r_\beta(x,u), \qquad  \sqrt{\beta}\, b_\beta(x,u),\qquad \sigma_\beta(x,u)
\]
are non-increasing as $\beta\geq 1$ increases, and also that ${\sigma_\beta}{\sigma_\beta}^T(x,u)\geq \lammin\I_d$ is preserved, and $\lambda/\beta\leq \lambda$. 

Recall that the regularity estimate \eqref{4.44} is independent of $\lambda$. So, the regularity estimate for $G_{\lambda/\beta}$ is independent of $\lambda$ and is only depending on the $L^\infty$ and Lipschitz continuity of the coefficients. Therefore,  we obtain  
\[
\|w\|_{C^{2,\alpha}}\leq C(\|w\|_\infty+\| r_\beta\|_\infty+\| r_\beta\|_\L)\leq C/{\beta}
\]
for some $C$ depending only on $d$ and the $L^\infty$ and Lipchitz norms of 
$ r_\beta(x,u)$, $b_\beta(x,u)$ and $\sigma_\beta(x,u)$ (in $x$),
but is  independent of $\lambda\in (0,\lambda_0)$ and $\beta\geq 1$.
This yields the conclusion on $V$ after transforming $w$ back to $V$.

Finally, it follows from \eqref{eq:regfeedbackC} that ${\pi}^*(x,u)$ is uniformly bounded and Lipschitz continuous in $x$, with a bound of the form $\exp(C/\lambda)$, where $C$ depends on $d$ and the constants in {\rm (H2)}.
\end{proof}

Lastly, we estimate the regularity of $V[\pi]$ given a bounded feedback control $\pi$.

\begin{lemma}[Regularity of ${V[\pi]}$]\lb{L.2.8}
Suppose that $\pi$ is uniformly bounded. Then we have for all $\lambda\in (0,\lambda_0)$ and $\beta\geq 1$, 
\[
\|V[\pi](\cdot)\|_\infty \leq \lambda \ln\|\pi\|_\infty/\beta,\quad \|\nabla V[\pi](\cdot)\|_\infty \leq C/\sqrt{\beta},
\]
where $C>0$ depends only on $d,\|\Sigma\|_{C^1}, \|r\|_\infty,\|b\|_\infty$ and $\lambda_0\| \int_{ U}  \pi\ln \pi\, du\|_\infty$.
Furthermore, if $\pi(x,u)$ is uniformly Lipschitz in $x$, then $V\in C^2$.
\end{lemma}
\begin{proof}
It follows from dynamic programming that $V=V[\pi]$ satisfies
\[
\begin{aligned}
&-\beta {V} (x) +  \tilde r(x,\pi) +\tilde  b(x,\pi) \cdot \nabla {V} (x) + \frac{1}{2}\tr(\Sigma(x)\nabla^2{V} (x)) - \lambda \int_{ U}  \pi(x,u)\ln \pi(x,u) du = 0.
\end{aligned}
\]
The comparison principle immediately yields that $\|V\|_\infty \leq \lambda \ln\|\pi\|_\infty/\beta$.

{Since $\Sigma(\cdot)\in C^1$}, $\tilde r(\cdot,\pi),\tilde b(\cdot,\pi)$ and $\pi(\cdot, u) $ are bounded, then by the classical regularity theory for elliptic equations (see e.g., \cite[Theorem 15.1]{ladyzhenskaia1968linear}), $V$ is uniformly bounded in $C^{1,\alpha}$. Applying the scaling argument the same as in the proof of Lemma \ref{L.3.5} yields the second estimate.

If $\pi(\cdot,u)$ is Lipschitz continuous, then so are $\tilde r(\cdot,\pi),\tilde b(\cdot,\pi)$ and $\int_{ U}  \pi(\cdot,u)\ln \pi(\cdot,u) du$. It follows from the  classical theory for elliptic equations that $V\in C^{2}$ .
\end{proof}

\section{Error estimates}

One of the major goal of the section is to estimate the difference between the two value functions from \eqref{2.2} and \eqref{eq:regVdis}, with a fixed bounded feedback randomized policy.

\begin{theorem}[Discretization error for fixed policy]\lb{P.3.3}
Assume {\rm (H1)(H2)}.
For $h\in (0,\frac12)$ and a given feedback randomized policy $\pi$, let $V[{\pi}]$ and $V_h[{\pi}]$ be defined by \eqref{2.2} and \eqref{eq:regVdis}, respectively. Further assume that $V[\pi](\cdot)\in C^2$ and
\[
\|{\pi}(\cdot,\cdot)\|_\infty\leq L_1,\,\quad \|\nabla V[{\pi} ](\cdot)\|_\infty\leq L_2,\quad \text{ for some $L_1,L_2\geq 2$}.
\]
Then there exists $C>0$ depending only on $d,b$ and $\Sigma$ such that
\[
\left\|V[{\pi} ]-V_h[{\pi}]\right\|_\infty\leq C(L_1L_2+\lambda L_1\ln L_1)h|\ln h|.
\]
\end{theorem}


\begin{proof}
Let us write $V(x)=V[\pi](x)$ for simplicity and recall the process $Y_t^\pi$ from \eqref{2.3}.
We start from It\^{o}'s formula applied to $V(Y^{\pi}_t)$ with discount factor $e^{-\beta t}$. For $h\in (0,1)$ and $i\in\N_0$, we have
\begin{align*}
e^{-\beta h} V(Y^{\pi}_{(i+1)h}) 
&= V(Y_{ih}^{\pi}) + \int_{ih}^{(i+1)h} e^{-\beta (t-ih)}\Big[-\beta  V(Y^{\pi}_t)+ b(Y^{\pi}_t,a_i)\cdot \nabla V(Y^{\pi}_t) \\
&\quad +  \frac{1}{2} \Sigma(Y^{\pi}_t):D^2 V(Y^{\pi}_t) \Big]\,dt+\int_{ih}^{(i+1)h} e^{-\beta (t-ih)}\sigma(Y_t^{\pi}) \cdot \nabla V(Y_t^{\pi})\, dB_t,
\end{align*}
where $a_i(\omega)=\varphi(Y_{ih}^{\pi},\xi^i(\omega))$. 
Then taking expectations gives
\beq\lb{7.2}
\begin{aligned}
\mathbb{E}\!\left[e^{-\beta  h} V(Y^{\pi}_{(i+1)h})\right]
&= \E\left[V(Y^{\pi}_{ih})\right] + \mathbb{E} \int_{ih}^{(i+1)h} e^{-\beta (t-ih)} 
   \Big[ -\beta V(Y^{\pi}_t) +  b(Y^{\pi}_t,a_i)\cdot \nabla V(Y^{\pi}_t) \\
&\qquad + \frac{1}{2} \Sigma(Y^{\pi}_t) :D^2 V(Y^{\pi}_t) \Big]\,dt.
\end{aligned}
\eeq

We divide the rest of the proof into several steps.

\medskip

\noindent{\bf Step 1.} In this step, we provide several useful estimates for the p.d.f for the stochastic process $Y_t^\pi$. Let $\rho(t,y)=\rho_{0,x}(t,y)$ satisfy $\rho(0,y) = \delta_x(y)$ for some fixed $x\in\R^d$, and for $t\in [ih,(i+1)h]$,
it satisfies the Fokker-Planck equation
\begin{equation*}
    \partial_t \rho(t,y) = \nabla\cdot\l[-b(y,a_{ih})\rho(t,y) + \frac12\nabla\cdot\l(\sigma(y)\sigma(y)^\top\rho(t,y)\r)\r].
\end{equation*}
Since $b(y,\cdot)$ and $\sigma(y)$ are assumed to be uniformly bounded in $C^{2,\alpha}$ in $y$, it follows from \cite[Theorem 2.12]{menozzi2021density} and its Remark 2.13 (see also \cite{Friedman-book-06,stroock1997multidimensional}) that for all $t\in (0,2]$: there exist $C\geq 1$ and $c\in (0,1]$ depending only on $d$ and the constants in (H2) such that
\begin{align}
\lb{den1'}
C^{-1} K(t, x-y) 
&\le \rho(t,y)
\le C K(t,x-y),  \\ 
\lb{den2'}
|\nabla_y^{\, j} \rho(t,y)| 
&\le C t^{-j/2} K(t, x-y), 
\qquad j = 1,2,
\end{align}
where $
K(t,x) := t^{-d/2} \exp\!\left( - {c|x|^2}/{t} \right)$.

Direct computation yields
\[
\rho(t,y)-\rho(ih,y)=\int_{ih}^t \rho_t(s,y)ds=\int_{ih}^t b(y,a_{ih}) \cdot\nb_x \rho+ \frac12\Sigma(y) : \nb^2_x\rho\, ds.
\]
Hence, we get for $t\in [ih,(i+1)h]$ and $h\leq ih\leq 2$,
\beq\lb{3.8}
\begin{aligned}
\int_{\R^d}|\rho(t,y)-\rho(ih,y)|dy&\leq \int_{\R^d}\int_{ih}^{(i+1)h}  \|b\|_\infty \|\nb_x \rho\|_\infty+ \frac12\|\Sigma\|_\infty\| \nb^2_x\rho\|_\infty dsdy\\
&\stackrel{\eqref{den2'}}\leq C\int_{\R^d}\int_{ih}^{(i+1)h}(\|b\|_\infty s^{-\tfrac12} +\|\Sigma\|_\infty s^{-1}) s^{-\tfrac{d}2} \exp\!\left( - \tfrac{c|x-y|^2}{s} \right)dsdy\\
&\leq C(\|b\|_\infty i^{-\tfrac12}h^{\tfrac12} +\|\Sigma\|_\infty i^{-1}).    
\end{aligned}
\eeq

To consider large $t\geq 2$, for any positive integer $k$, let $i_k:= \lceil k/h\rceil$ and so $i_kh\in [k,k+h)$. Recall the reproduction property in \cite[Theorem 1.4]{menozzi2021density} item 6. Then for all $t\geq i_kh$, we have
\beq\lb{3.5}
\rho(t,y)=\int_{\R^d}\rho(i_k h, z)\rho_{i_kh,z}(t,y) dz
\eeq
where $\rho_{i_kh,z}(t,y)$ satisfies for $i\geq i_k$ and then $t\in [ih, (i+1)h]$,
\begin{equation*}
    \partial_t \rho_{i_kh,z}(t,y) = \nabla\cdot\l[-b(y,a_{ih})\rho_{i_kh,z}(t,y) + \frac12\nabla\cdot\l(\sigma(y)\sigma(y)^\top\rho_{i_kh,z}(t,y)\r)\r]
\end{equation*}
with $\rho_{i_kh,z}(i_kh,y)=\delta_z(y)$.

Now suppose $t\in [k+1,k+2]$ for some $k\geq 1$, then $t-i_k h\in [1-h,2]$. It follows from \eqref{den2'} that
\begin{equation}\lb{3.6}
|\nabla_y^{\, j} \rho_{i_kh,z}(t,y)| 
\;\le\; C (t-i_kh)^{-j/2} K(t-i_kh, z-y)\leq C  K(t-i_kh, z-y), 
\qquad j = 1,2.
\end{equation}
Also using that $\rho(i_kh,\cdot)$ is a probability density function, for $j=1,2$ and $t\in [k+1,k+2]$, we have 
\begin{align*}
\int_{\R^d}|\nabla^j\rho(t,y)|dy&\stackrel{\eqref{3.5}}\leq \iint_{\R^{2d}}\rho(i_kh,z) |\nabla^j \rho_{i_kh,z}(t,y)|dzdy\\
&\stackrel{\eqref{3.6}}\leq C\int_{\R^d}\rho(i_kh,z) \left[\int_{\R^d}K(t-i_kh,z-y)dy\right] dz\leq C,    
\end{align*}
where $C$ is independent of $k$. So, the estimate holds uniformly for all $t\geq 2$.

Finally, it follows that for $t\in [ih,(i+1)h]$ with $ih\geq 2$, we have
\beq\lb{3.8'}
\begin{aligned}
\int_{\R^d}|\rho(t,y)-\rho(ih,y)|dy&\leq \int_{\R^d}\int_{ih}^{(i+1)h}  \|b\|_\infty \|\nb_x \rho\|_\infty+ \frac12\|\Sigma\|_\infty\| \nb^2_x\rho\|_\infty dsdy\\
&\leq C(\|b\|_\infty+\|\Sigma\|_\infty) \int_{ih}^{(i+1)h}\int_{\R^d}|\nabla_x\rho(s,y)|+|\nabla^2_x\rho(s,y)|dyds\\
&\leq C(\|b\|_\infty+\|\Sigma\|_\infty)h .    
\end{aligned}
\eeq

\smallskip

\noindent{\bf Step 2.} In this step we aim to bound
\[
D_{i,f}(t):=\E\l[\int_ U f(Y^{\pi}_{ih},u)- f(Y^{\pi}_t,u) du\r]
\]
where $t\in [ih, (i+1)h]$ and $f(x,u)$ is a bounded function. 
First, direct computation yields
\beq\lb{D01}
D_{i,f}(t)\leq 2\|f\|_\infty.
\eeq

Now, we seek for a more accurate bound.
If $i\in [1,2/h]$ and  $t\in [ih,(i+1)h]$, we have
\beq\lb{D02}
\begin{aligned}
 |D_{i,f}(t)|&= \l| \iint_{\R^d\times U} f (y,u) (\rho(t,y)-\rho(ih,y))dy du\r|\\
 &\leq \|f\|_\infty \iint_{\R^d\times U}|\rho(t,y)-\rho(ih,y)| dy du\\
  &\stackrel{\eqref{3.8}}\leq C \|f\|_\infty  (\|b\|_\infty i^{-\tfrac12}h^{\tfrac12} +\|\Sigma\|_\infty i^{-1}).
\end{aligned}
\eeq
If $i\geq 2/h$ and  $t\in [ih,(i+1)h]$,
\beq\lb{D03}
\begin{aligned}
|D_{i,f}(t)|&\leq  \|f\|_\infty \iint_{\R^d\times U}|\rho(t,y)-\rho(ih,y)| dy du\\
&\stackrel{\eqref{3.8'}}\leq C \|f\|_\infty  (\|b\|_\infty  +\|\Sigma\|_\infty )h.
\end{aligned}
\eeq

Now, we set
\begin{align*}
D_{i,\pi}(t):=\E\l[\lambda\int_ U {\pi}(Y_t^{\pi},u)\ln {\pi}(Y_t^{\pi},u)du-\lambda\int_ U {\pi}(Y_{ih}^{\pi},u)\ln {\pi}(Y_{ih}^{\pi},u)du\r],
\end{align*}
and then
\[
D_{i,\pi}(t)=\E\l[\int_ U \lambda{\pi}(y,u)\ln {\pi}(y,u)(\rho(t,y)-\rho(ih,y))dydu\r].
\]
Let us take $f(x,u)= \lambda\pi(x,u)\ln \pi(x,u)$, and since $\pi\leq L_1$ and $L_1\geq 2$, we have $\|f\|_\infty\leq \lambda L_1\ln L_1$. We get from \eqref{D01}--\eqref{D03} that for $t\in [ih,(i+1)h]$,
\beq\lb{Dpi}
\begin{aligned}
|D_{i,\pi}(t)|\leq 
\begin{cases}
    2\lambda L_1\ln L_1, & i=0,\\
  C\lambda L_1\ln L_1 \Big(\|b\|_\infty i^{-\tfrac12}h^{\tfrac12} +\|\Sigma\|_\infty i^{-1}\Big), & i\in [1,2/h],\\
  C \lambda L_1\ln L_1 (\|b\|_\infty  +\|\Sigma\|_\infty )h,& i> 2/h.  
\end{cases}
\end{aligned}
\eeq

Next, we define
\[
D_{i,r}(t):=\E\l[ r(Y^{\pi}_{ih},a_{ih})-\tilde r(Y^{\pi}_t,{\pi})\r].
\]
By the definition, $\xi^i$ is independent of $Y_{ih}^{\pi}$ and $\omega\to u=a_i(\omega)$ has the distribution density of ${\pi}(Y_{ih}^{\pi},u)$. Therefore
\begin{align*}
D_{i,r}(t)&=\E\l[\int_ U r(Y^{\pi}_{ih},u)\pi(Y^{\pi}_{ih},u)- r(Y^{\pi}_t,u) \pi(Y^{\pi}_t,u)du\r]\\
&=\E\l[\int_ U r(y,u) {\pi}(y,u)(\rho(t,y)-\rho(ih,y))dydu\r].    
\end{align*}
Setting $f(x,u)=r(x,u)\pi(x,u)$, then $\|f\|_\infty\leq L_1\|r\|_\infty$. It follows from \eqref{D01}--\eqref{D03} that for $t\in [ih,(i+1)h]$,
\beq\lb{Dr}
\begin{aligned}
|D_{i,r}(t)|\leq 
\begin{cases}
    2L_1\|r\|_\infty, & i=0,\\
  CL_1\|r\|_\infty \Big(\|b\|_\infty i^{-\tfrac12}h^{\tfrac12} +\|\Sigma\|_\infty i^{-1}\Big), & i\in [1,2/h],\\
  C L_1\|r\|_\infty (\|b\|_\infty  +\|\Sigma\|_\infty )h,& i> 2/h.  
\end{cases}
\end{aligned}
\eeq

\noindent{\bf Step 3.}
The goal is to estimate 
\[
D_{i,b}(t):=\E[ b(Y^{\pi}_t,a_i)\cdot \nabla V(Y^{\pi}_t)]-\E[\tilde b(Y^{\pi}_t,{\pi})\cdot \nabla V(Y^{\pi}_t)].
\]
Firstly, since $b$ and $\nabla V$ are uniformly finite, we have a trivial bound
\[
|D_{i,b}(t)|\leq 2L_2\|b\|_\infty.
\]

To obtain a more accurate estimate, we define
\[
D_{i,b,1}(t):= \E\l[ \tilde b(Y^{\pi}_t,{\pi})\cdot \nabla V(Y^{\pi}_t)-\tilde b(Y^{\pi}_{ih},{\pi})\cdot \nabla V(Y^{\pi}_{ih})\r],
\]
\[
 D_{i,b,2}(t):=\E\l[ b(Y^{\pi}_t,a_i)\cdot \nabla V(Y^{\pi}_t)-\tilde b(Y^{\pi}_{ih},{\pi})\cdot \nabla V(Y^{\pi}_{ih})\r].
\]
If $i\in [1,2/h]$ and  $t\in [ih,(i+1)h]$, we have from \eqref{D02} that
\[
\begin{aligned}
 |D_{i,b,1}(t)|
  &= \l| \iint b(y,u)\pi(y,u)\nabla V(y) (\rho(t,y)-\rho(ih,y))dy du\r|\\
   &\leq CL_1L_2\|b\|_\infty   (\|b\|_\infty i^{-\tfrac12}h^{\tfrac12} +\|\Sigma\|_\infty i^{-1}).
\end{aligned}
\]
While if $i\geq 2/h$, we use \eqref{D03} to get
\[
\begin{aligned}
 |D_{i,b,1}(t)| 
  \leq C L_1L_2\|b\|_\infty     (\|b\|_\infty  +\|\Sigma\|_\infty )h.
\end{aligned}
\]

Next, similarly as before, since $\xi^i$ is independent of $Y_{ih}^{\pi}$ and $\omega\to u=a_i(\omega)$ has the distribution density of ${\pi}(Y_{ih}^{\pi},u)$, we get
\begin{align*}
 D_{i,b,2}(t) 
  &=  \E\l[ b(Y^{\pi}_t,a_i)\cdot \nabla V(Y^{\pi}_t)-b(Y^{\pi}_{ih},a_i)\cdot \nabla V(Y^{\pi}_{ih})\r]\\
  &=\iint b(y,a_i)\nabla V(y) (\rho(t,y)-\rho(ih,y))dy du.
\end{align*}
Then, it follows from \eqref{D02} and \eqref{D03} that for $i\in [1,2/h]$ and  $t\in [ih,(i+1)h]$,
\[
\begin{aligned}
 |D_{i,b,2}(t)| \leq C\|b\|_\infty L_1 (\|b\|_\infty i^{-\tfrac12}h^{\tfrac12} +\|\Sigma\|_\infty i^{-1}),
\end{aligned}
\]
and for $i\geq 2/h$, 
\[
\begin{aligned}
 |D_{i,b,2}(t)| \leq C \|b\|_\infty L_1 (\|b\|_\infty  +\|\Sigma\|_\infty )h.
\end{aligned}
\]

Overall, combining the above estimates yields for  $t\in [ih,(i+1)h]$,
\beq\lb{3.11}
\begin{aligned}
|D_{i,b}(t)|&\leq |D_{i,b,1}(t)|+|D_{i,b,2}(t)|\\
&\leq 
\begin{cases}
    2L_2\|b\|_\infty, & i=0,\\
  CL_1L_2\|b\|_\infty   (\|b\|_\infty i^{-\tfrac12}h^{\tfrac12} +\|\Sigma\|_\infty i^{-1}), & i\in [1,2/h],\\
  C L_1L_2\|b\|_\infty  (\|b\|_\infty  +\|\Sigma\|_\infty )h,& i> 2/h.  
\end{cases}.
\end{aligned}
\eeq

\medskip

\noindent{\bf Step 4.} We complete the proof in this step.
Let us denote
\[
\begin{aligned}
A_i &:=
e^{-\beta (i+1) h}\mathbb{E}\!\left[ V(Y^{\pi}_{(i+1)h})\right]
- e^{-\beta ih}\E[ V(Y^{\pi}_{ih})] \\
&+ \mathbb{E} \int_{ih}^{(i+1)h}  e^{-\beta t}\left[
    \tilde r(Y^{\pi}_{ih},{\pi})-\lambda\int_ U {\pi}(Y^{\pi}_{ih},u) \ln {\pi}(Y^{\pi}_{ih},u) \right]dudt.
\end{aligned}
\]
Since $V$ satisfies the following HJB equation, 
\[
-\beta V(Y^{\pi}_t)+  \tilde b(Y^{\pi}_t,{\pi})\cdot \nabla V(Y^{\pi}_t)  + \frac{1}{2} \Sigma(Y^{\pi}_t) :D^2 V(Y^{\pi}_t)=-\tilde r(Y^{\pi}_t,{\pi})+\lambda\int_ U  {\pi}(Y^{\pi}_t,u)\ln {\pi}(Y^{\pi}_t,u) du,
\]
we can write
\[
\begin{aligned}
A_i &=
e^{-\beta (i+1)h}\mathbb{E}\!\left[ V(Y^{\pi}_{(i+1)h})\right]
- e^{-\beta ih}\E[ V(Y^{\pi}_{ih})] \\
&+\mathbb{E} \int_{ih}^{(i+1)h}  
  e^{-\beta t} \lambda  \l[ \int_ U  {\pi}(Y^{\pi}_t,u)\ln {\pi}(Y^{\pi}_t,u) du-\int_ U  {\pi}(Y^{\pi}_{ih},u) \ln {\pi}(Y^{\pi}_{ih},u) \r]dudt\\
 &  - \mathbb{E} \int_{ih}^{(i+1)h} e^{-\beta t} 
   \Big[ -\beta V(Y^{\pi}_t)+  \tilde b(Y^{\pi}_t,{\pi})\cdot \nabla V(Y^{\pi}_t)  + \frac{1}{2} \Sigma(Y^{\pi}_t) :D^2 V(Y^{\pi}_t) \Big]\,dt\\
&+\mathbb{E} \int_{ih}^{(i+1)h}  e^{-\beta t}\left[
    \tilde r(Y^{\pi}_{ih},{\pi})-\tilde r(Y^{\pi}_{t},{\pi})\right]dt.
\end{aligned}
\]
Applying \eqref{7.2} to the above yields
\begin{align*}
A_i&=\mathbb{E} \int_{ih}^{(i+1)h}  
  e^{-\beta t} \lambda  \l[ \int_ U  {\pi}(Y^{\pi}_t,u)\ln {\pi}(Y^{\pi}_t,u) du-\int_ U  {\pi}(Y^{\pi}_{ih},u) \ln {\pi}(Y^{\pi}_{ih},u) \r]dudt\\
  &+\mathbb{E} \int_{ih}^{(i+1)h} e^{-\beta t} 
  \Big[ b(Y^{\pi}_t,a_i)- \tilde b(Y^{\pi}_t,{\pi})\Big]\cdot \nabla V(Y^{\pi}_t)  \,dt+\mathbb{E} \int_{ih}^{(i+1)h}  e^{-\beta t}\left[
    \tilde r(Y^{\pi}_{ih},{\pi})-\tilde r(Y^{\pi}_{t},{\pi})\right]dt\\
  &=\mathbb{E} \int_{ih}^{(i+1)h}  
  e^{-\beta t} \left[D_{i,\pi}(t)+D_{i,b}(t)+D_{i,r}(t)\right]dt  .    
\end{align*}

Therefore, we get from the estimates in the previous steps (\eqref{Dpi}, \eqref{Dr}, \eqref{3.11}):
\beq\lb{7.4}
\begin{aligned}
|A_i|\leq 
\begin{cases}
   2 (L_2\|b\|_\infty +\lambda L_1\ln L_1+L_1\|r\|_\infty)h, & i=0,\\
  C(L_1L_2\|b\|_\infty  +\lambda L_1\ln L_1+L_1\|r\|_\infty) (\|b\|_\infty i^{-\tfrac12}h^{\tfrac12} +\|\Sigma\|_\infty i^{-1}), & i\in [1,2/h],\\
  Ce^{-\beta ih}(L_1L_2\|b\|_\infty  +\lambda L_1\ln L_1+L_1\|r\|_\infty) (\|b\|_\infty  +\|\Sigma\|_\infty )h,& i> 2/h.  
\end{cases}
\end{aligned}
\eeq
Hence, summing up $A_i$ for $i=0,1,\ldots$ yields
\beq\lb{3.22}
\left|\sum_{i=0}^\infty A_i\right|\leq C_*\left[h+C\int_h^{2} (ht^{-\tfrac{1}{2}}+ht^{-1}) dt+C\int_2^\infty e^{-\beta t}h\, dt\r]\leq C_*h|\ln h|,
\eeq
where 
\[
C_*:=C(L_1L_2+\lambda L_1\ln L_1)
\]
with $C$ depending only on $d$, $b$ and $\Sigma$.

Finally, note that
\[
\sum_{i=0}^\infty A_i=-\E[V(Y_0^\pi)]+\sum_{i=0}^\infty\int_{ih}^{(i+1)h} e^{-\beta t} 
   \Big[  \tilde r(Y^{\pi}_t,{\pi})-\lambda\int_ U  {\pi}(Y^{\pi}_{ih},u) \ln {\pi}(Y^{\pi}_{ih},u) du\Big]dt.
\]
Recall that $Y_0^{\pi}=x$ (so $\E[V(Y_0^\pi)]=V[\pi](x)$). 
Thus, also using
\[
\sum_{i=0}^\infty\int_{ih}^{(i+1)h} \left|e^{-\beta t} -e^{-\beta ih}\right|
   \Big|  \tilde r(Y^{\pi}_t,{\pi})-\lambda\int_ U  {\pi}(Y^{\pi}_{ih},u) \ln {\pi}(Y^{\pi}_{ih},u) du\Big|dt\leq C(\|r\|_\infty+\lambda L_1\ln L_1)h,
\]
we then obtain from \eqref{eq:regVdis} and \eqref{3.22} that
\[
\left|V[{\pi}](x)-V_h[{\pi}](x)\right|\leq C_*h|\ln h|.
\]
\end{proof}

In the following corollary,  we plug-in the optimal randomized feedback policy from the MDP to the continuous time problem. Then the error w.r.t. the optimal value function of the MDP is $h|\log h|$. 
The same estimate is derived if we plug in the optimal control obtained from the continuous-time problem to the discrete-time problem, making the control piece-wise constant in time and the value function is a sum over discrete-time.

Recall $V$ and $\pi^*$ in \eqref{2.2} and \eqref{eq:regfeedbackC}, and recall $V_h$ and $\pi_h^*$ from Proposition \ref{T.2.3}. We have  $V=V[\pi^* ]$ and $V_h(x)=V_h[\pi_h^*](x)$.

\begin{corollary}[Discretization error for value functions]\lb{C.3.4}
Assume {\rm (H1)(H2)}, $h\in(0,\tfrac12)$, and 
\beq\lb{cond00}
C_d\|\nabla\sigma\|_\infty h^\frac12\leq \frac12\quad\text{and}\quad\beta \geq 1 + 2\|\nabla b\|_\infty + \frac{\|\nabla \Sigma\|^2_\infty}{4\lambda_{\min}}
\eeq
for some dimensional constant $C_d$.
Then there exists $C>0$ such that
\[
\left\|V[ \pi^*_h ](\cdot)-V_h(\cdot)\right\|_\infty\leq C(1+\| \pi^*_h\|_\infty) h|\ln h|\leq C\exp(\tfrac{C}\lambda)h|\ln h|,
\]    
and
\[
\left\|V(\cdot)-V_h[\pi^*](\cdot)\right\|_\infty\leq C(1+\|\pi^*\|_\infty) h|\ln h|\leq C\exp(\tfrac{C}{\lambda})h|\ln h|.
\] 
The constant $C$ here depends only on $d$, {\rm (H1)(H2)} and an upper bound of $\lambda>0$.
\end{corollary}
\begin{proof}
Note that Lemma  \ref {L.2.4} (and its remark) and Lemma \ref{L.2.8} show that the conditions of Theorem \ref{P.3.3} hold with $\pi= \pi^*_h$. Indeed, we have
\[
        \ll  \pi^*_h(x,u) \rl_\infty \leq \exp\l(C/{\lambda}\r).
\]
And then since
\[
\left|\lambda \int_{ U}   \pi^*_h(x,u)\ln  \pi^*_h(x,u) du\right|\leq C\lambda+\lambda \ln\left\| \pi^*_h(x,u) \right\|_\infty\leq C\lambda+C,
\]
Lemma \ref{L.2.8} yields $\|\nabla V[ \pi^*_h ](\cdot)\|\leq C$ where $C$ depends on the assumptions and an upper bound of $\lambda$.
Because $V_h=V_h[ \pi^*_h]$, we conclude the proof for the first estimate by applying Theorem \ref{P.3.3}.

Similarly,  by Lemma \ref{L.3.5} and \eqref{eq:regfeedbackC}, there exists $C$ independent of $\lambda\in (0,\lambda_0)$ such that 
$\|\nabla V\|_\infty\leq C$ and $\|\pi^*\|_\infty\leq \exp(C/\lambda)$.
Then the second estimate follows from Theorem \ref{P.3.3}. 
\end{proof}

In the following result, we show that the optimal feedback control for the discrete relaxed control problem is close to the one in the continuous setting.

\begin{theorem}[Discretization error for optimal feedback controls]\lb{T.3.4}
Under the assumptions of Corollary \ref{C.3.4}, there exists a constant $C > 0$, depending on the same constants as in the corollary, such that
\begin{align*}
V[ \pi^*_h ](\cdot)&\leq V(\cdot)\leq  
V[ \pi^*_h ](\cdot)+C\exp(\tfrac{C}{\lambda}) h|\ln h|,\\
V_h[\pi^*](\cdot)&\leq V_h(\cdot)\leq V_h[\pi^*](\cdot)+C\exp(\tfrac{C}{\lambda}) h|\ln h|,\\
|V(\cdot)- V_h(\cdot)|&\leq C\exp(\tfrac{C}{\lambda}) h|\ln h|.
\end{align*}
Actually, in all the above estimates, we can replace $C\exp(\tfrac{C}{\lambda})$ by $C(1+\|\pi^*\|_\infty+\| \pi^*_h\|_\infty)$.
\end{theorem}
The first inequality tells us the following ``error''. If we obtain the optimal control from the RL approximation, and then plug it back to the true continuous-time problem, the error we expect in terms of the value function.

The second estimate is useful, for instance, when our goal is to solve an RL problem by approximating it with a PDE. It shows that if we apply the optimal control from the continuous-time problem to the discrete-time formulation, the resulting error with respect to the true optimal value function in RL can be quantified. 
\begin{proof}
For the first estimate, notice that
\[
V(x)=\sup_{\pi}V[\pi ](x)=V[\pi^* ](x)
\]
by \eqref{eq:regV}.
Therefore, it follows from the definition that
\[
V(x)\geq V[\pi^*_h  ](x).
\]

Next, write $C_*:=C\exp(\tfrac{C}{\lambda})$. We apply Corollary \ref{C.3.4} to get that
\[
|V[\pi^*_h  ](x)-V_h(x)|\leq C_* h|\ln h|,
\]
and Corollary \ref{C.3.4} to get that
\[
\left|V(x)-V_h[\pi^*](x)\right|\leq  C_* h|\ln h|.
\]
By the definition of $ \pi^*_h$, we have
\[
V_h=V_h(\pi^*_h)\geq V_h[\pi^*].
\]
The three estimates yield
\beq\lb{333}
V[\pi^*_h  ]\geq V_h(\pi^*_h)-C_* h|\ln h|\geq  V_h[\pi^*]\geq V(\cdot)-2C_*h|\ln h|.
\eeq

The second claim follows similarly.
For the last estimate, \eqref{333} yields
\[
V(\cdot)\leq V_h(\cdot)+C_* h|\ln h|.
\]
The other direction follows from a similar argument. 

Finally, the statement of Corollary \ref{C.3.4}  yields that we can replace $C_*$ by $C(1+\|\pi^*\|_\infty+\| \pi^*_h\|_\infty)$. 
\end{proof}

\section{Further regularity properties of MDP}


In this section, we provide several regularity estimates for MDP.

For a fixed $u$ $\in$ $U$, recall that $\rho_t(y|x,u)$ is defined in \eqref{def of rho_h}.  
For any bounded function $W:\R^d\to\R$ and $t\geq 0$, define
\[
(P_t^u W(\cdot))(x) := \int_{\R^d} W(y) \rho_t(y|x,u) dy.
\]
Since $\rho_t(\cdot|x,u)$ is a probability distribution, for any bounded $W$,
\beq\lb{P.0}
|P_t^u W(x)|\leq \|W\|_\infty.
\eeq
By the definition of $P_t^u$ and that $b,\sigma$ are time-independent, we know that $f(x,t):=P_t^uW (x)$ satisfies the following forward Kolmogorov equation \cite{pavliotis2016stochastic},
\beq\lb{feqn}
    \pt_t f(x,t) =\mL_{b,\Sig}f(x,t), \quad \text{with }f(x,0) = W(x), 
\eeq
and
\beq\lb{mL}
\mL_{b,\Sig} := b(x,u) \cdot\nb_x + \frac12\Sigma(x) : \nb^2_x.
\eeq

The next two lemmas are concerned with the regularity of solutions to \eqref{feqn} with possibly a source term. We aim to provide a quantitative tracking of the constants involved.

\begin{lemma}[Estimates for \eqref{feqn}]\label{lemma:bd}
Assume {\rm (H2)}. Let $W$ be a uniformly bounded and Lipschitz continuous function on $\R^d$ and let $f$ be as the above. For any $t\geq 0$, we have
\[
   \ll \nb_x f(x,t) \rl_\infty \leq e^{A_0t} \ll \nb_x W \rl_\infty, \quad A_0: = 2\ll \nb b \rl_\Linf  + \tfrac{\ll \nb\Sig \rl^2_\Linf}{4\lammin}.
\]
\end{lemma}
\begin{proof}
Let $q_l: = \pt_{x_l}f$ for $l=1,\ldots,d$ and $q=(q_l)$ and write $\ll q \rl_\infty^2 = \sum_{l=1}^d (\pt_{x_l}f(x,t))^2$. Due to \eqref{feqn}, $q_l$ satisfies   
\[
\pt_t q_l = \mL_{b,\Sig}q_l + \mL_{\pt_{x_l}b,\pt_{x_l}\Sig} (f), \quad \text{with }q_l(0,x) = \pt_{x_l}W(x).
\]
Multiplying $q_l$ to the above equation and then summing it over $l$ give
\[
\begin{aligned}
\pt_t (\frac12 \ll q \rl_\infty^2) &= \mL_{b,\Sig} (\frac12  \ll q \rl_\infty^2) - \frac12\sum_l (\nb q_l)^\top\Sig (\nb q_l) + q^\top \nb b \cdot q + \sum_l\frac12(\pt_{x_l}\Sig:\nb q) q_l\\
&\leq \mL_{b,\Sig} (\frac12  \ll q \rl_\infty^2) - \frac\lammin2\ll\nb q\rl_\infty^2  + \ll \nb b \rl_\infty\ll q \rl_\infty^2  + \frac12\ll \nb\Sig \rl_\infty\ll\nb q\rl_\infty\ll q \rl_\infty\\
&\leq \mL_{b,\Sig} (\frac12  \ll q \rl_\infty^2)   + \l(\ll \nb b \rl_\Linf  + \frac{\ll \nb\Sig \rl^2_\Linf}{8\lammin}+\frac12\r)\ll q \rl_\infty^2. 
\end{aligned}
\]
Then for 
\[
A_0=2\ll \nb b \rl_\Linf  + \frac{\ll \nb\Sig \rl^2_\Linf}{4\lammin}\quad{and}\quad w:=\frac{1}{2}\exp(-A_0t)\|q\|_\infty^2,
\]
we have
$
\pt_t w\leq \mL_{b,\Sig} w$.
Since $\frac{1}{2}\|\nabla W\|_\infty$ is a solution to $\pt_t v= \mL_{b,\Sig} v$ and $w(\cdot,0)\leq \frac{1}{2}\|\nabla W\|_\infty$, the comparison principle yields
\[
\frac{1}{2}\exp(-A_0t)\|q\|_\infty^2\leq \frac{1}{2}\|\nabla W\|_\infty,
\]
which shows the conclusion.
\end{proof}

The next goal is to obtain the derivative of $f$ with respect to $u$.

\begin{lemma}[Higher estimates for \eqref{feqn}]\label{lemma:bd of cond-exp}
Assume {\rm (H2)}. Let $g(x,t)$ satisfy
\beq\lb{feqn1}
    \pt_tg(x,t) =\mL_{b,\Sig}g(x,t)+R(x,t), \quad \text{with }g(x,0) = 0.
\eeq
Then there exists $C\geq 1$ depending only on $d,\lambda_{\min}$ and $M_1$ such that
for any $t\geq 0$ we have
\[
\|g(\cdot,t)\|_\infty\leq t\sup_{s\in [0,t]}\|R(\cdot,s)\|_\infty,\qquad \|\nabla_x g(\cdot,t)\|_\infty\leq C\sqrt{t}\sup_{s\in [0,t]}\|R(\cdot,s)\|_\infty.
\]

In particular, if letting $f^u(x,t):=(P_t^uW(\cdot))(x)$ for some bounded function $W$, then
\[
\|\nabla_u f^u(\cdot,t)\|_\infty\leq \|\nabla_u b\|_\infty t \sup_{s\in [0,t],u\in U}\|\nabla_x f^\mru(\cdot,s)\|_\infty,
\]
\[
\ \|\nabla_u\nabla_x f^u(\cdot,t)\|_\infty\leq C\|\nabla_u b\|_\infty\sqrt{t} \sup_{s\in [0,t],u\in U}\|\nabla_x f^\mru(\cdot,s)\|_\infty  .
\]
\end{lemma}
\begin{proof}
Since $b,\sigma$ are uniformly bounded and Lipschitz continuous by (H2), the classical results (see e.g., \cite{Friedman-book-06,menozzi2021density,stroock1997multidimensional}) yield that
the density function $\rho_t(y|x,u)$ satisfies the following estimates for all $t\in (0,1]$: there exist $C\geq 1$ and $c\in (0,1]$ depending only on $d$ and the constants in (H2) such that
\begin{equation}\lb{den2}
|\nabla_x^{\, j} \rho_t(y|x,u)| 
\;\le\; C t^{-j/2} K(t, x-y), 
\qquad j = 1,2,
\end{equation}
where
\begin{equation*}
K(x,t) := t^{-d/2} \exp\!\left( - {c|x|^2}/{t} \right),
\qquad \; t>0.
\end{equation*}
Consequently, for any bounded function $f$ we have
\beq\lb{P.1}
|\nabla_x (P_t^uf)(x)|\leq \int_{\R^d} f(y) |\nabla_x\rho_t(y|x,u)| dy\leq Ct^{-\frac12}\int_{\R^d} f(y) K(t,x-y) dy\leq Ct^{-\frac12}\|f\|_\infty.
\eeq
By Duhamel's principle, we have
$
g(\cdot,t) = \int_0^t P^u_{t-s}R(\cdot,s) ds$. Therefore, by \eqref{P.0},
\[
|  g(x,t) |= |\int_0^t  P^u_{t-s}R(x,s) ds|\leq \int_0^t \|R(\cdot,s)\|_\infty ds\leq t\sup_{s\in [0,t]}\|R(\cdot,s)\|_\infty,
\]
and by \eqref{P.1},
\[
|\nabla_x  g(x,t) |= |\int_0^t \nabla_x P^u_{t-s}R(x,s) ds|\leq C\int_0^t (t-s)^{-\frac12}\|R(\cdot,s)\|_\infty ds\leq C\sqrt{t}\sup_{s\in [0,t]}\|R(\cdot,s)\|_\infty.
\]

For the last claim, let us take $u_1\neq u$. Set $w(x,t):=f^u(x,t)-f^{u_1}(x,t)$, which then satisfies
\[
 \pt_tw(x,t) =\mL_{b,\Sig}w(x,t)+(b(x,u)-b(x,u_1))\cdot\nabla w(x,t), \quad \text{with }w(x,0) = 0.
\]
Here we used that $\Sigma$ is independent of $u$. It follows from the first two claims of the lemma  that
\begin{align*}
\| w(x,t)\|_\infty&\leq t\sup_{s\in [0,t]}\|(b(\cdot,u)-b(\cdot,u_1))\cdot\nabla w(\cdot,s)\|_\infty\\
&\leq 2t|u-u_1|\|\nabla_u b\|_\infty \sup_{s\in [0,t],v\in U}\|\nabla f^{v}(\cdot,s)\|_\infty ,  
\end{align*}
\begin{align*}
\|\nabla w(x,t)\|_\infty&\leq C\sqrt{t}\sup_{s\in [0,t]}\|(b(\cdot,u)-b(\cdot,u_1))\cdot\nabla w(\cdot,s)\|_\infty\\
&\leq 2C\sqrt{t} |u-u_1|\|\nabla_u b\|_\infty\sup_{s\in [0,t],v\in U}\|\nabla f^{v}(\cdot,s)\|_\infty ,  
\end{align*}
which finishes the proof.
\end{proof}

It is known that $V_h\in L^\infty$ from  Proposition \ref{T.2.3}. In the following result, we provide an explicit estimate and, moreover, we show that it is Lipschitz continuous.

\begin{lemma}[Regularity of $V_h$]\label{lemma:bd of v vx}
Under the assumptions of {\rm (H1)(H2)}, for all $h\in (0,1]$ 
we have
\[
\begin{aligned}
    &\ll V_h(\cdot) \rl_\infty \leq \frac{h\ll r \rl_\infty}{1-\exp(-\beta h)},
\end{aligned}
\]
and, if $\beta\geq 1+A_0$ with $A_0=2\ll \nb b \rl_\Linf  + \frac{\ll \nb\Sig \rl^2_\Linf}{4\lammin}$, we have
\[
\ll\nb_xV_h(\cdot) \rl_\infty\leq   e^h\ll \nb_x r \rl_\infty,
\]
and for some $C$ depending only on the assumptions,
\[
\| \nb^2_x V_h(\cdot)\|_\infty \leq Ch^{-1}.
\]
\end{lemma}
\begin{proof}
By the contraction property from Lemma \ref{dis-opt-contraction}, we get
\[
\begin{aligned}
    \ll V_h \rl_\infty = \ll T^* V_h - T^* 0  +     T^*0 \rl_\infty \leq \ll T^* V_h - T^* 0 \rl +     \ll T^*0 \rl \leq \gamma \ll V_h  \rl + \ll T^*0\rl_\infty. 
\end{aligned}
\]
Applying Lemma \ref{dis-operator} and (H1), the above inequality leads to
\[
\begin{aligned}
    \ll V_h \rl_\infty \leq \frac1{1-\gamma} \ll T^*0\rl_\infty = \frac{\lam h}{1-\gamma} \ln  \int_{U}\exp\l(\frac{r(x,u)}{\lam } \r)du \leq \frac{ h \ll r \rl_\infty}{1-\gamma},  
\end{aligned}
\]
which gives the first inequality. 

For the second inequality, recall that $V_h$ is the unique fixed point of $T^*V=V$, and
\[
T^*V(x) =  \lambda h \ln \l[\int_{U} \exp\l(\frac{r(x,u)h + \gamma\int_{\R^d} V(y) \rho_h(y|x,u) dy}{\lambda h}\r) du\r].
\]
For any Lipschitz function $V$ and a fixed control $u$, denote
\[
Q(x,u) := r(x,u)h + \gamma \int_{\R^d} V(y)\rho_h(y|x,u) dy \quad\text{and}\quad \pi(x,u) :=\frac{\exp\l(\frac{Q(x,u)}{\lam h}\r)}{\int_{U} \exp\l(\frac{Q(x,u')}{\lam h}\r)  du'}.
\]
Then one has
\[
\begin{aligned}
     \nb_x (T^* V) &=  \frac{\lam h\nb_x\l[\int_{U} \exp\l(\frac{Q(x,u)}{\lam h}\r)  du \r]}{\int_{U} \exp\l(\frac{Q(x,u)}{\lam h}\r)  du}  
    = \E_{u\sim \pi(x,\cdot)}\l[ \nb_xQ(x,u) \r] \\
    &\leq   \ll \nb_x Q \rl_\infty \leq \ll \nb_xr\rl_\infty h +  \gamma  \ll\nabla_x  (P_h V) \rl_\infty\\
    &\leq \ll \nb_xr\rl_\infty h +  \gamma \exp\left( A_0h\right) \ll \nb_x V \rl_\infty
\end{aligned}
\]
where we applied Lemma \ref{lemma:bd} in the last inequality. 
Therefore, since $\gamma=e^{-\beta h}$, by approximating $V_h$ by smooth functions, we actually obtain
\[
\begin{aligned}
&\ll \nb V_h \rl_\infty = \ll \nb_x (T^* V_h)\rl_\infty
\leq \exp\left[\l(A_0-\beta\r)h\right]  \ll  \nb V_h \rl_\infty +  \ll \nb_x r \rl_\infty h 
\end{aligned}
\]
which, by the assumption that $\beta\geq1+A_0$,  implies that for $h\in (0,1]$,
\[
\begin{aligned}
&\ll \nb V_h \rl_\infty \leq \frac{h}{1-e^{(A_0-\beta )h}} \ll \nb_x r \rl_\infty\leq e^h\ll \nb_x r \rl_\infty .
\end{aligned}
\]

\def\V{\mathbb{V}}
Now, for any $V\in C^2(\R^d)$, direct computation yields
\[
\begin{aligned}
     \nb^2_x (T^* V) &= \E_{u\sim \pi(x,\cdot)}\l[ \nb^2_xQ(x,u) \r] + \frac{1}{\lam h}\l(\E_{u\sim \pi(x,\cdot)}[\nb_xQ(x,u)^2] - \l(\E_{u\sim \pi(x,\cdot)}[\nb_xQ(x,u)]\r)^2]\r) .
\end{aligned}
\]
Let us denote
\[
\V_{u\sim \pi(x,\cdot)}[\nb_xQ(x,u)]:=\E_{u\sim \pi(x,\cdot)}[\nb_xQ(x,u)^2] - \l(\E_{u\sim \pi(x,\cdot)}[\nb_xQ(x,u)]\r)^2].
\]
One has
\[
\ll \V_{u\sim \pi(x,\cdot)}[\nb_xQ(x,u)] \rl_\infty \leq \frac{|U|^2}4 \ll \nb_u\nb_xQ(x,u) \rl^2_\infty
\]
It follows from Lemma \ref{lemma:bd of cond-exp} that
\[
\ll \nb_u\nb_xQ(x,u) \rl_\infty = \ll  \nb_u\nb_x r  \rl_\infty h + \gamma \ll  \nb_u\nb_x (P_h^u V) \rl_\infty \leq C\l[h + \gamma\ll \nb_x V \rl_\infty  \sqrt{h}\r].
\]
Therefore, one has
\[
\ll \V_{u\sim \pi(x,\cdot)}[\nb_xQ(x,u)] \rl_\infty \leq  C^2\l[h^2 + \gamma^2\ll \nb_x V \rl_\infty^2 h\r].
\]

Furthermore, due to \eqref{den2},
\[
|\nabla_x^2 (P_h V)(x)|\leq \int_{\R^d} V(y) |\nabla^2_x\rho_t(y|x,u)| dy\leq Ct^{-1}\int_{\R^d} V(y) K(t,x-y) dy\leq Ct^{-1}\|V\|_\infty.
\]
Thus, we get
\[
\E_{u\sim \pi(x,\cdot)}\l[ \nb^2_xQ(x,u) \r]  \leq \ll \nb_x^2 r \rl_\infty h + \gamma \ll \nb_x^2 P_hV \rl_\infty \leq \ll \nb_x^2 r \rl_\infty h + Ch^{-1}\ll V \rl_\infty,
\]
which implies that
\[
 \ll \nb^2 T^* V \rl \leq \ll \nb_x^2 r \rl_\infty h + Ch^{-1}\ll V \rl_\infty  + \frac1\lam  C^2\l[h + \gamma^2\ll \nb_x V \rl_\infty^2  \r].
\]
Hence, since $ V_h =T^*V_h$, we get
\[
\ll \nb_x^2 V_h \rl_\infty \leq \l(\ll \nb_x^2 r \rl_\infty +  \frac1\lam C^2\r)h  +Ch^{-1}\ll V_h \rl_\infty+ \frac1\lam C^2\gamma^2\ll \nb_x V_h \rl_\infty^2. 
\]

\end{proof}

Lemma \ref{L.2.4} obtains an explicit $L^\infty$-bound for $ \pi^*_h$.
In the following
lemma, we show that $ \pi^*_h$ is uniformly Lipschitz continuous with Lipschitz constant of order $\tfrac{C}{\lambda\sqrt{h}}$.

\begin{proposition}[Regularity of $\pi_h^*$] \label{prop:gradpi}
    Under the assumptions of Lemma \ref{lemma:bd of v vx}, we have
    \[
    \begin{aligned}
        \ll  \nb_x (\ln \pi^*_h ) \rl_\infty \leq \frac{C}{\lam \sqrt{h}}   ( \|\nabla_u\nabla_x r\|_\infty \sqrt{h} + \|\nabla_ub\|_\infty\|\nabla_xr\|_\infty ).
    \end{aligned}
    \]
\end{proposition}

\begin{proof}
 Denote
\[
Q^*(x,u) := r(x,u) h + \gamma \int_{\R^d} V_h(y)\rho_h(y|x,u) dy
\]
\[
\text{and}\quad Z(x):= \int_{U}\exp\l(\frac{r(x,u)h + \int V_h(y) \rho_h(y|x,u) dy}{\lambda}\r)du.
\]
It follows from Lemma \ref{lemma:bd of cond-exp} that 
\[
\big\|\nabla_u\nabla_x \int_{\R^d} V_h(y)\rho_h(y|x,u) dy\big\|_\infty\leq C\|\nabla_u b\|_\infty\sqrt{h} \sup_{s\in [0,h],u\in U}\|\nabla_x P^\mru_s V_h(\cdot)\|_\infty  .
\]
By Lemma \ref{lemma:bd} and Lemma \ref{lemma:bd of v vx}, we have
\[
\sup_{s\in [0,h],u\in U}\|\nabla_x P^\mru_s V_h(\cdot)\|_\infty\leq e^{A_0h}\|\nabla_x  V_h(\cdot)\|_\infty\leq \frac{he^{A_0h}}{1-e^{(A_0-\beta )h}} \ll \nb_x r \rl_\infty.
\]
So, we obtain
\beq\lb{Q8}
\|\nabla_u\nabla_xQ^*\|_\infty\leq \|\nabla_u\nabla_x r\|_\infty h+  \frac{C\|\nabla_u b\|_\infty e^{A_0h}h}{1-e^{(A_0-\beta )h}} \ll \nb_x r \rl_\infty \sqrt{h} .
\eeq

Finally, since
    \[
    \begin{aligned}
        \nb_x \pi^*_h (x,u) & = \frac{\exp\l(\frac{ Q^*(x,u)}{\lam h}\r)\frac{ \nb_xQ^*(x,u)}{\lam h}}{Z(x)}- \frac{\exp\l(\frac{ Q^*(x,u)}{\lam h}\r)\nb Z(x)}{Z(x)^2}\\
    & = \frac{\pi^*_h (x,u)}{\lam h} \l( \nb_xQ^* - \lam h\frac{\nb Z}{Z}\r) = \frac{\pi^*_h (x,u)}{\lam h} \l(\nb_xQ^* - \E_{u\sim \pi^*_h }[\nb_xQ^*]\r),
    \end{aligned}
    \]
we get from \eqref{Q8} that
\[
\begin{aligned}
        \ll \nb_x (\ln \pi^*_h ) \rl_\infty  &\leq (\lam h)^{-1}  \l\|\nb_xQ^* - \E_{u\sim \pi^*_h }[\nb_xQ^*]\r\|_\infty\\
        &\leq (\lam h)^{-1}  |U| \ll \nb_u\nb_x Q^* \rl_\infty\\
        &\leq (\lam \sqrt{h})^{-1}  |U| C ( \|\nabla_u\nabla_x r\|_\infty \sqrt{h} + \|\nabla_ub\|_\infty\|\nabla_xr\|_\infty ).
\end{aligned}
\]
\end{proof}

\section{Convergence to the classical optimal control}

In this section, we compare the classical optimal control and the relaxed optimal control problems, both in the continuous setting. We further need the following condition.
\begin{enumerate}
    \item[(H3)] Assume $ U\subseteq \bbR^N$ with some $N\geq 1$ and  \cite[Assumption 4.2]{pia2022}. In particular, $ U=[0,1]^N$ satisfies the condition. 
\end{enumerate}

To state the precise assumption, for any $\alpha \in (0,1]$ and $e\in\bbS^{N-1}$, we define a cone $ \Delta^\alpha_{e}$ as 
$$
\Delta^\alpha_{e} := \{x \in \mathbb{R}^N : x\cdot e\geq \alpha |x|\}.
$$
Then \cite[Assumption 4.2]{pia2022} says that: there exist $\zeta > 0$ and $\alpha \in (0, 1]$ such that for any $u \in  U$, there is a cone $\Delta^\alpha_{e}$ for some $e=e_u\in\bbS^{N-1}$ such that
\[
\Delta^\alpha_{e}(u)\cap B_{\zeta}(u)\subseteq  U,
\]
where
$\Delta^\alpha_{e}(u):=\{u+x:x\in \Delta^\alpha_{e}\}$.


As a consequence of this condition, we obtain the following lemma, which is used to provide a polynomial upper bound of $\pi^*$ in \eqref{eq:regfeedbackC}. While the proof closely follows the one in \cite[Lemma 4.4]{pia2022}, our claim is slightly more general; we therefore provide a complete proof for clarity and self-containment.

\begin{lemma}[An $L^\infty$ estimate]\lb{L.4.1}
Assuming {\rm (H3)}, let $g:\bbR^d\times\bbR^d\times U\to \bbR$ be a function that is uniformly Lipschitz continuous. Then there exists $C$ depending only on  {\rm (H3)} such that
\begin{equation*}
\Gamma(x, p, u) :=\frac{\exp\left(g(x,p,u) \right)}{\int_{ U} \exp\left(g(x,p,u) \right) du}\leq C (1+\|\nabla_u\, g\|_\infty)^N.
\end{equation*} 
\end{lemma}
\begin{proof}
Fixing any $u_0 \in U$, direct computation yields
\begin{equation*}
\Gamma(x, y, u_0)^{-1} = {\int_U \exp\left(g(x,p,u)-g(x,p,u_0)\right) du} \ge {\int_U \exp\left(-\|\nabla_u\, g\|_\infty|u - u_0|\right) du}.
\end{equation*}
Then, by the assumption, there exists $e_u\in\bbS^N$ such that
\begin{equation*}
\int_U e^{-\|\nabla_u\, g\|_\infty|u-u_0|} du \ge \int_{\Delta^\alpha_{e_{u_0}}(u_0) \cap B_\zeta(u_0)} e^{-\|\nabla_u\, g\|_\infty|u-u_0|} du = \int_{\Delta^\alpha_{e_{1}} \cap B_\zeta(0)} e^{-\|\nabla_u\, g\|_\infty|u|} du, 
\end{equation*}
where, in the last equality, we shift and rotate the cone $\Delta^\alpha_{e_{u_0}}(u_0)$ to $\Delta^\alpha_{e_{1}}$ with $e_1$ denoting the positive direction of the first coordinate of $\bbR^N$. 

It remains to estimate the right-hand side of the above. If $\|\nabla_u\, g\|_\infty \zeta \le 1$,
\beq\lb{4.9}
\int_{\Delta^\alpha_{e_1} \cap B_\zeta(0)} e^{-\|\nabla_u\, g\|_\infty|u|} du \ge \text{Leb}(B_\zeta(0)).
\eeq
If $\|\nabla_u\, g\|_\infty\zeta > 1$, we use polar coordinates to get
\begin{align*}
\int_{\Delta^\alpha_{e_1} \cap B_\zeta(0)} e^{-\|\nabla_u\, g\|_\infty|u|} du &= C_1 \int_0^\zeta r^{N-1} e^{-\|\nabla_u\, g\|_\infty r} dr\\
&= C_1 \|\nabla_u\, g\|_\infty^{-N} \int_0^{\|\nabla_u\, g\|_\infty\zeta} z^{N-1}e^{-z} dz \ge C_1 C_2 \|\nabla_u\, g\|_\infty^{-N},
\end{align*}
where $C_1$ is the surface area of $\bbS^{N-1}$ and $C_2:=\int_0^1 z^{N-1}e^{-z} dz$. This and \eqref{4.9} finish the proof.
\end{proof}

Denote ${\pi}^{*,\lambda}:=\pi^*$ from \eqref{eq:regfeedbackC}, and  ${\pi}_h^{*,\lambda}:=\pi^*_h$ from \eqref{2.16}. Assuming {\rm (H3)} and using Lemma \ref{L.4.1}, we obtain a polynomial upper bound (in terms of $\lambda$) for ${\pi}^{*,\lambda}_h$ and ${\pi}^{*,\lambda}$, which improves the bound of $\exp(C/\lambda)$ that we used before.

\begin{lemma}[Upper bound for $\pi_h^{*,\lambda}$]\lb{L.6.2}
Assuming  {\rm (H2)(H3)}, we have
\[
{\pi}^{*,\lambda}_h(x,u),\quad {\pi}^{*,\lambda}(x,u)\leq  C(1+\lambda^{-N})
\]
\end{lemma}
\begin{proof}
    First, let us derive the upper bound for $\pi^{*,\lambda}_{h}$. It follows from \eqref{dis-opt-control} that
\[
    \pi^{*,\lambda}_{h}(x,u) = \frac{\exp\left(g(x,u) \right)}{\int_{ U} \exp\left(g(x,u) \right) du}, \quad\text{ where }
g(x,u):=\frac{r(x,u)h + \gamma\int  V_h(y) \rho_h(y|x,u) dy}{\lambda h}.
\]
By Lemmas \ref{lemma:bd} -- \ref{lemma:bd of v vx}, we get if $\beta\geq 1+2\ll \nb b \rl_\Linf  + \frac{\ll \nb\Sig \rl^2_\Linf}{4\lammin}=:1+A_0$,
\begin{align*}
\left|\int  V_h(y) \rho_h(y|x,u) dy\right|&=\|\nabla_u P_h^uV_h\|_\infty \\
&\stackrel{{\rm Lemma}\, \ref{lemma:bd of cond-exp}}\leq \|\nabla_u b\|_\infty h \sup_{s\in [0,h],u\in U}\|\nabla_x P_s^uV_h\|_\infty\\ 
&\stackrel{{\rm Lemma}\, \ref{lemma:bd}}\leq 
\|\nabla_u b\|_\infty \exp(A_0h) h \|\nabla_x V_h\|_\infty
\\ 
&\stackrel{{\rm Lemma}\, \ref{lemma:bd of v vx}}\leq     \|\nabla_u b\|_\infty \exp(A_0h) e^hh \ll \nb_x r \rl_\infty\leq Ch.
\end{align*}
Hence, we obtain
$|\nabla_u\, g(x,u)|\leq \tfrac{C}{\lambda}.$ We then get from Lemma \ref{L.4.1} that
\beq\lb{pidis}
\pi^{*,\lambda}_{h}(x,u)\leq C (1+\|\nabla_u\, g\|_\infty)^N\leq C(1+\lambda^{-N}).
\eeq

Next, note that
\[
{\pi}^{*,\lambda}(x,u)=\Gamma(x,\nabla V(\cdot),u)
\]
\[
\text{ with }
\Gamma(x,p,u) := \frac{\exp\left(\frac{1}{\lambda} \left[r(x,u) + b(x,u) \cdot p\right] \right)}{\int_{ U} \exp\left(\frac{1}{\lambda} \left[r(x,u) + b(x,u) \cdot p \right] \right) du},
\]
where $V=V[{\pi}^{*,\lambda} ]$ solves \eqref{eq:HJBreg}. By Lemma \ref{L.3.5}, there exists $C\geq 1$ depending only on $d$, $\lambda_0$ and the constants in {\rm (H2)} (independent of $\beta\geq 1$ and $\lambda\in (0,\lambda_0)$) such that $\|\nabla V\|_{\infty}\leq C/\sqrt{\beta}$. Hence, Lemma \ref{L.4.1} implies 
\[
{\pi}^{*,\lambda}(x,u)\leq C(1+\tfrac1\lambda[\|\nabla_u r\|_\infty+\|\nabla_u b\|_\infty\|\nabla V\|_\infty])^N\le C(1+\lambda^{-N})
\]
for some $C$ depending only on $d,\lambda_0$ and the conditions.
\end{proof}

Combining Theorem \ref{T.3.4} with the lemma, we obtain the following corollary.

\begin{corollary}
\lb{C.6.2}
Under the assumptions of Theorem \ref{T.3.4} and {\rm (H3)} and $\lambda\in (0,\lambda_0)$, we have
\begin{align*}
\|V[ \pi^*_h ](\cdot)- V(\cdot)\|_\infty+
\|V_h[\pi^*](\cdot)- V_h(\cdot)\|_\infty+
\|V(\cdot)- V_h(\cdot)\|_\infty\leq C\lambda^{-N} h|\ln h|.
\end{align*}
\end{corollary}

\vspace{1mm}

For a general feedback randomized policy  ${\pi}:\bbR^d\to \mathcal{P}( U)$, let 
\begin{equation}\lb{5.2}
v[{\pi}](x) = \mathbb{E}\bigg[ \int_0^{\infty} e^{-\beta t}   \tilde r(X^{{\pi}}_t, {\pi}(X_t^{{\pi}},\cdot) ) dt \bigg| X^{{\pi}}_0 = x \bigg],
\end{equation}
where $X^{{\pi}}_t$ satisfies \eqref{eq:expS} with $\varpi_t(u):={\pi}(X^\pi_t,u)$ and $\tilde r$ is from \eqref{eq:tilder}. 

The following result measures the difference between the optimal feedback controls for the classical problem and the relaxed problem as $\lambda\to 0$. Recall $V[{\pi} ]$ from \eqref{2.2}. 

\begin{theorem}[Continuous feedback controls]\lb{T.5.2}
Assume {\rm (H2)(H3)}. There exists $C$  depending only on $d$, {\rm (H2)(H3)} and an upper bound of $\lambda$ and $1/\beta$ such that
\[
\left|v[{\pi}^{*,\lambda}](\cdot)-V[{\pi}^{*,\lambda} ](\cdot)\right|\leq C\lambda (1+|\ln \lambda|).
\]
Moreover,
letting $v$ solve \eqref{eq:HJBclassical}, we have
\[
\left|v(\cdot)-v[{\pi}^{*,\lambda}](\cdot)\right|\leq C\lambda (1+|\ln \lambda|).
\]
\end{theorem}
\begin{proof}
By \eqref{5.2} and \eqref{2.2}, 
\beq\lb{5.3}
v[{\pi}](x)-V[\pi](x)=\mathbb{E}\bigg[  \lambda \int_{ U} {\pi}(X_t^{{\pi}},u) \ln {\pi}(X_t^{{\pi}},u) du  \bigg) dt \bigg| X^{{\pi}}_0 = x \bigg].
\eeq
Thus, to prove the first estimate, it suffices to show that
\[
\int_ U {\pi}^{*,\lambda}(x,u) \ln{\pi}^{*,\lambda}(x,u)du \leq C(1+|\ln\lambda|).
\]
Indeed, it follows from Lemma \ref{L.6.2} that
\[
\int_ U {\pi}^{*,\lambda}(x,u) \ln{\pi}^{*,\lambda}(x,u)du \leq \ln C(1+\lambda^{-N})\int_ U {\pi}^{*,\lambda}(x,u) du \leq C'(1+|\ln\lambda|)
\]
which finishes the proof of the first statement.

For the second one, it follows from \cite[Corollary 16]{TZZ} and the remark after that 
\[
|v(\cdot)-V[{\pi}^{*,\lambda} ](\cdot)|\leq C\lambda(1+|\ln\lambda|).
\]
This and the first claim imply the second claim.
\end{proof}

The following result says that given time discretization $h$, in order to approximate the classical optimal control problem ($\lambda=0$), we might select an optimal temperature parameter $\lambda$ depending on the dimension $N$ of the action space. Indeed, if selecting $\lambda:=h^{1/(N+1)}$, the error is bounded by
\[
C\lambda |\ln \lambda|+C\lambda^{-N} h|\ln h|\leq C h^{\tfrac{1}{N+1}}|\ln h|.
\]


\begin{theorem}[MDP feedback controls]\lb{T.6.3}
 Assume {\rm (H2)(H3)}, $h\in(0,\tfrac12),\lambda_0>0$, and \eqref{cond00}.
There exists $C$  depending only on $d$, {\rm (H2)(H3)} and $\lambda_0$ such that, for any $\lambda\in (0,\lambda_0)$, we have
\begin{align*}
    \left|v(\cdot)-V[{\pi}_h^{*,\lambda} ](\cdot)\right|&\leq C\lambda |\ln \lambda|+C\lambda^{-N} h|\ln h|,\\
\left|v(\cdot)-v[{\pi}_h^{*,\lambda}](\cdot)\right|&\leq C\lambda |\ln \lambda|+C\lambda^{-N} h|\ln h|.
\end{align*}   
\end{theorem}

\begin{remark}[Comparison to \cite{zhu2025optimal}]\label{rmk: compare phibe}
This is the main theorem of the paper. It shows that if one computes the optimal feedback control from the regularized MDP formulation \eqref{dis-be-soft}, then deploys this policy in the true continuous-time dynamics and evaluates performance using the continuous-time objective, the resulting policy is suboptimal compared to the true optimal control. The performance gap is bounded by
\[
\tilde{O}(\lambda + \lambda^{-N} h).
\]
Optimizing over the regularization parameter \(\lambda\) yields the bound \(\tilde{O}(h^{\frac{1}{N+1}})\). This can be interpreted as the best achievable rate for any regularized RL algorithms.

The MDP-based framework relies solely on the discrete-time transition distribution \(\rho_h(y \mid x,u)\) defined in \eqref{def of rho_h} to approximate the optimal feedback control associated with the continuous-time HJB equation \eqref{eq:HJBclassical}. In contrast, \cite{zhu2025optimal} propose an alternative approximation, termed PhiBE, which uses the same transition distribution \(\rho_h(y \mid x,u)\) but achieves an \(O(h)\) convergence rate, strictly faster than that of the MDP framework. This is because the PhiBE formulation mirrors the structure of the (unregularized) HJB equation and explicitly leverages the SDE-driven nature of the underlying dynamics. By comparison, the MDP formulation does not incorporate this structural information, even though it is used in the analysis. 
\end{remark}
\color{black}

\begin{proof}
It follows from Lemma \ref{L.6.2} that
\[
\pi^{*,\lambda}(x,u),\quad \pi^{*,\lambda}_{h}(x,u)\leq C (1+\|\nabla_u\, g\|_\infty)^N\leq C(1+\lambda^{-N}).
\]
The first claim follows from 
\[
\left|v(\cdot)-V[{\pi}^{*,\lambda} ](\cdot)\right|\leq C\lambda (1+|\ln \lambda|)\quad\text{ and }
\]
\begin{align*}
V[\pi^{*,\lambda}_{h} ](\cdot)\leq V[{\pi}^{*,\lambda} ](\cdot)\leq  
V[\pi^{*,\lambda}_{h} ](\cdot)+C(1+\lambda^{-N}) h|\ln h|,
\end{align*}
which are due to Theorem \ref{T.5.2} and Theorem \ref{T.3.4}, respectively.

For the second claim, it remains to bound
$
v[\pi^{*,\lambda}_{h}](x)-V[\pi^{*,\lambda}_{h} ](x)$.
Due to \eqref{5.3}, we have
\beq\lb{5.5}
\|v[\pi^{*,\lambda}_{h}]-V[\pi^{*,\lambda}_{h} ]\|_\infty\leq \lambda \sup_{x\in\bbR^d} \left|\int_ U {\pi}_h^{*,\lambda}(x,u) \ln{\pi}_h^{*,\lambda}(x,u)du\right|.
\eeq
It follows from \eqref{pidis} that
\[
\left|\int_ U {\pi}_h^{*,\lambda}(x,u) \ln{\pi}_h^{*,\lambda}(x,u)du\right|\leq C(1+|\ln\lambda|)
\]
which, combining with \eqref{5.5}, yields the second claim. 
\end{proof}

\section*{Acknowledgements}
Y. Zhu is supported by the NSF grants No 2529107. Y. P. Zhang acknowledges support from NSF CAREER grant DMS-2440215 and Simons Foundation Travel Support MPS-TSM-00007305.

\appendix

\section{Remarks}

In this section, we discuss two explicit continuous stochastic optimal control examples. The first one shows that the feedback control of a classical optimal control problem is, in general, discontinuous. While the one for the relaxed control problem is continuous. The second example shows that such discontinuity of feedback controls might lead to completely different dynamics if one applies the it discretely in time.

\medskip

The first one is a temperature control problem \cite{GXZ20}. Consider
\begin{equation}
\label{eq:bang}
\begin{aligned}
v(x) :=  \inf_{(\nu_t)}  \, \mathbb{E}\bigg[\int_0^\infty e^{- t} f(X_t) dt \bigg], 
\end{aligned}
\end{equation}
where $X_t$ satisfies
\[
dX_t = -\nabla f(X_t) dt + \sqrt{2 \nu_t} dB_t, \quad X_0 = x,
\]
and $(\nu_t, \, t \ge 0)$ is taken as the control over $ U=[a,1]$ with $a\in (0,1)$.
For $b(x,u) = -\nabla f(x)$, and $\sigma(x,u) = \sqrt{2u}$, $v$ satisfies
\begin{equation}
\label{eq:HJBbang}
- v(x) + f(x) - \nabla f(x) \cdot \nabla v(x)+ \inf_{u \in [a,1]} \left[u \tr(\nabla^2v(x)) \right] = 0.
\end{equation}
From this, we get that the optimal feedback policy $\mru^*$ can be represented by
\[
\mru^*(x) :=
\left\{ \begin{array}{lcl}
a & \mbox{if } \tr(\nabla^2v(x)) \ge 0, \\
1 & \mbox{if } \tr(\nabla^2v(x))< 0, \\
\end{array}\right.
\]
which is a discontinuous function.

The exploratory version  of \eqref{eq:bang} with temperature parameter $\lambda$ is introduced in \cite{GXZ20}, and the corresponding  exploratory HJB equation for the value function $V$ is given by
\begin{equation}
\label{eq:ellipticPDE}
-\rho V_{}(x)  + \nabla f(x) \cdot \nabla V_{}(x)+ f(x) - \lambda \ln \int_a^1 \exp\left(- \frac{\tr(\nabla^2 V_{}(x))}{\lambda}  u \right) du = 0,
\end{equation}
with the optimal feedback control 
\begin{equation*}
\pi^*(x,u) = {\exp\left(- \frac{\tr(\nabla^2 V_{}(x))}{\lambda}  u \right)}\left[{\int_a^1 \exp\left(- \frac{\tr (\nabla^2 V_{}(x))}{\lambda}  u \right) du}\right]^{-1}, \quad u \in [a,1].
\end{equation*}
So for the relaxed problem, the feedback control as a distribution is continuous, by Lemma \ref{L.3.5}.

\medskip

In the next example, we show that the dynamics are not stable with respect to the discretization.
We consider $\sigma\equiv 0$ and both the reward function $r^h(x)$ and the optimal value function $v^h(x)$ depend on $h\in (0,1)$.  
Let us take $\beta>0$, $\gamma\in \bbR$ and $n\geq 2$, and consider the following equation in dimension $1$:
\beq\lb{4.1}
-\beta v^h+\sup_{u\in [-1,1]}\left[r^h(x,u)+u v^h_x\right]=0, \qquad x \in \R, 
\eeq
where 
\[
r^h(x,u):=\beta (\gamma x+h^{n}\sin\frac{2\pi x}h)-\gamma u-2\pi  h^{n-1}| \cos\frac{2\pi x}{h}|.
\]
By direct computation,
\[
v^h(x):=\gamma x+h^{n}\sin\frac{2\pi x}h
\]
is the solution to \eqref{4.1} and $h^{n}\sin\frac{2\pi x}h$ can be viewed as a perturbation when $h$ is small.
Moreover, we have
\beq\lb{4.11}
\begin{aligned}
r^h(x,u)+u v^h_x
&=\beta v^h+2\pi h^{n-1}(u\cos\frac{2\pi x}{h}-| \cos\frac{2\pi x}{h}|).    
\end{aligned}
\eeq
And we claim that the following holds:
\begin{itemize}
    \item The $C^n$ norm of $v^h$ and the Lipschitz norm of $r^h$ are locally uniformly finite independent of $h\in (0,1)$.

    \item The system states in the discrete setting (with step size $h$) and in the continuous setting can be completely different even as $h\to 0$. 
\end{itemize}

In the following, we compute the dynamics of the system states in both  the discrete and the continuous setting. Using \eqref{4.1} and \eqref{4.11}, the optimal feedback control is given by
\[
 \mru^*(x):=
\begin{cases}
    1 &\text{ when }\cos \frac{2\pi x}{h}>0,\\
    -1&\text{ when }\cos \frac{2\pi x}{h}<0.
\end{cases}
\]
Let $Y$ solve iteratively for $t\in [ih,(i+1)h)$ with $i\geq 0$,
\beq\lb{ode1}
dY= \mru^*( Y(ih))dt\quad\text{ and }Y(0)=0.
\eeq
Since $ \mru^*(Y(0))= \mru^*(0)=1$, the ODE \eqref{ode1} yields
$Y(t)=t$ for $t\in [ih,(i+1)h]$. Next, for $t\in [h,2h]$, we note that
$\cos(\tfrac{2\pi Y(h)}h) =\cos2\pi=1$.
Thus, $ \mru^*(Y(h))=1$ and again the ODE \eqref{ode1} yields $Y(t)=t$ for $t\in [h,2h]$. Continuing in this manner, we obtain 
\beq\lb{Yt}
Y(t)=t\text{ for all $t\geq 0$},\quad\text{and}\quad   \mru^*(Y(ih))= \mru^*(ih)=1.
\eeq
However, the ODE $dX= \mru^*(X)dt$ becomes for all $i\in\mathbb{Z}$,
\[
dX=dt \text{ when }X(t)\in (ih-\tfrac{h}{4},ih+\tfrac{h}{4})
\]
\[
dX=-dt \text{ when }X(t)\in (ih+\tfrac{h}{4},ih+\tfrac{3h}{4}).
\]
Solving the ODE with initial data $X(0)=0$ yields
\[
X(t)=\left\{
\begin{aligned}
   & t-ih &&\text{ if }t\in (ih-\tfrac{h}{4},ih+\tfrac{h}{4}),\\
    &    ih+\tfrac{h}{2}-t &&\text{ if }t\in (ih+\tfrac{h}{4},ih+\tfrac{3h}{4}),
\end{aligned}
\right.
\]
which is completely different from $Y(t)$ in \eqref{Yt} in the discrete setting.
Thus, even just perturbing the system slightly (as $h\to 0$), the system states obtained from the classical optimal control problem \eqref{eq:classicalS} and \eqref{eq:classicalV} can be largely different. In this sense, the feedback control problem is not stable with respect to the discretization.

It remains unclear whether the value functions exhibit instability relative to the discretization; however, in this specific example, the resulting value functions do not differ significantly.

\section{Proof of the discrete-time regularized Bellman equation}\label{proof of r-be}
\begin{proposition}
\label{prop:DPP_regularized_discrete}
The discrete-time regularized Bellman equation \eqref{dis-be-soft} is equivalent to the regularized optimal control problem defined in \eqref{eq:regVdis} - \eqref{2.16}. 
\end{proposition}

\begin{proof}
\textbf{Step 1. Fixed feedback policy:}
Fix an admissible feedback randomized policy $\pi(x,\cdot)$. Denoting $r_h^\lambda(x,u;\pi):=r(x,u)-\lambda \log \pi( x,u))h$, by \eqref{eq:regVdis},
\[
\begin{aligned}
&V_h[\pi](x)
=\int_U \pi(x,u)
\Bigg(
r_h^\lambda(x,u;\pi)
+e^{-\beta h}\E\l[\sum_{j=0}^\infty e^{-\beta j h}
r_h^\lambda(Y_{(j+1)h}^{\pi},\nu_{(j+1)h};\pi) | Y_0^\pi = x, \nu_0 = u\r]
\Bigg)du\\
&=\int_U \pi(x,u)
\Bigg(
r_h^\lambda(x,u;\pi)
\\
&\hspace{1cm}
+e^{-\beta h}\E\l[ \E\l[\sum_{j=0}^\infty e^{-\beta j h}
r^\lambda(Y_{(j+1)h}^{\pi},\nu_{(j+1)h};\pi)|Y_h^\pi, Y_0^\pi = x, \nu_0 = u\r] | Y_0^\pi = x, \nu_0 = u\r]
\Bigg)du.\\
&=\int_U \pi(x,u)
\Bigg(
r_h^\lambda(x,u;\pi) +e^{-\beta h}\E\l[ V_h[\pi](Y^\pi_h)| Y_0^\pi = x, \nu_0 = u\r]
\Bigg)du,
\end{aligned}
\]
where 
the last equality is by the Markov property. 
Therefore, one has
\[
V_h[\pi](x)
=
\E_{u\sim\pi(x,\cdot)}\l[ {r}_h^\lambda(x,u;\pi)   + \gamma \E_{Y^{{\pi}}_{h}\sim\rho_h(\cdot|x,u)}\l[V_h(Y^{{\pi}}_{h}) \r]\r]
\]

Now we prove the {\bf contraction} property of the operator 
\begin{equation}\label{def of Tpi}
[T^\pi W](x) := \E_{u\sim\pi(x,\cdot)}\l[ r_h^\lambda(x,u;\pi)   + \gamma \E_{Y^{\pi}_{h}\sim\rho_h(\cdot|x,u)} \l[W(Y^{\pi}_{h}) \r]\r]    
\end{equation}
For any $x$,
\begin{equation}
\begin{aligned}
\bigl|T^\pi W_1(x) - T^\pi W_2(x)\bigr| &=
\bigg|
\mathbb E_{u\sim \pi(x,\cdot)}
\Big[
\gamma \mathbb E_{y\sim \rho_h(\cdot|x,u)}
\big(W_1(y) - W_2(y)\big)
\Big]
\bigg| \\
&\le
\gamma \,
\mathbb E_{u\sim \pi(x,\cdot)}
\mathbb E_{y\sim \rho_h(\cdot|x,u)}
\big|W_1(y) - W_2(y)\big| \le
\gamma \|W_1 - W_2\|_\infty.    
\end{aligned}    
\end{equation}
Taking supremum over $x$ on the LHS of the above inequality gives the $\gamma$-contraction:
\begin{equation} \label{contraction-fix-policy}
\|T^\pi W_1 - T^\pi W_2\|_\infty
\le
\gamma \|W_1 - W_2\|_\infty.
\end{equation}

Finally, we prove the {\bf monotonicity} of the operator. If $W_1(x) \leq W_2(x)$ for $\forall x\in \R^d$, then
\begin{equation}\label{mono}
\begin{aligned}
T^\pi W_1(x) &= \E_{u\sim\pi(x,\cdot)}\l[ r_h^\lambda(x,u;\pi)   + \gamma \E_{Y^{\pi}_{h}\sim\rho_h(\cdot|x,u)} \l[W_1(Y^{\pi}_{h}) \r]\r]  \\
&\leq \E_{u\sim\pi(x,\cdot)}\l[ r_h^\lambda(x,u;\pi)   + \gamma \E_{Y^{\pi}_{h}\sim\rho_h(\cdot|x,u)} \l[W_2(Y^{\pi}_{h}) \r]\r] = T^\pi W_2(x)    
\end{aligned}    
\end{equation}

\medskip

\noindent
\textbf{Step 2. Proof of the upper bound:}
Given $V_h(x)$ from \eqref{2.16}, we prove that $V_h(x) \leq T^*[V_h](x)$ with $T^*$ defined in \eqref{111}.

For any feedback randomized policy $\pi(x,\cdot)$, by Step 1, $V_h[\pi]= T^\pi V_h[\pi]$. By the monotonicity of $T^\pi$ \eqref{mono} and $V_h[\pi]\le V_h$ pointwisely, one has $V_h[\pi]= T^\pi V_h[\pi] \leq  T^\pi V_h \leq  T^* V_h $. Taking the supremum over $\pi$ on the RHS, one has, 
\begin{equation}\label{eq:step2}
\begin{aligned}
V_h=\sup_{\pi}V_h[\pi]\leq  T^*V_h, \quad \text{point-wisely}.
\end{aligned}
\end{equation}

\medskip
\noindent
\textbf{Step 3. Proof of the lower bound:}
Fix $\varepsilon>0$. By definition of $T^*$, there exists a feedback policy $\pi_\varepsilon$ such that
\[
T^* V_h(x)
\le
T^{\pi_\varepsilon} V_h(x) + \varepsilon,
\qquad \forall x.
\]
Let $V[\pi_\varepsilon]$ denote the value function under $\pi_\varepsilon$, which satisfies $V[\pi_\varepsilon] = T^{\pi_\varepsilon} V[\pi_\varepsilon]$.
By Step 2 \eqref{eq:step2}, the  monotonicity \eqref{mono}, and $\gamma$-contractivity \eqref{contraction-fix-policy} of $T^{\pi_\varepsilon}$, 
\begin{equation*}
\begin{aligned}
V_h - V[\pi_\varepsilon]& \leq T^*V_h -T^{\pi_\varepsilon} V[\pi_\varepsilon]\\
&\le
\|T^* V_h - T^{\pi_\varepsilon} V_h\|_\infty
+
\|T^{\pi_\varepsilon} V_h - T^{\pi_\varepsilon} V[\pi_\varepsilon]\|_\infty\\
&\le
\varepsilon + \gamma \|V_h - V[\pi_\varepsilon]\|_\infty. 
\end{aligned}    
\end{equation*}
Hence, $\|V_h - V[\pi_\varepsilon]\|_\infty\le
\varepsilon + \gamma \|V_h - V[\pi_\varepsilon]\|_\infty$, which gives $\|V_h - V[\pi_\varepsilon]\|_\infty
\le
\frac{\varepsilon}{1-\gamma}$.
Using again $T^* V_h \le T^{\pi_\varepsilon} V_h + \varepsilon$, we obtain for each $x$,
\begin{align*}
T^* V_h(x)
&\le
T^{\pi_\varepsilon} V_h(x) - T^{\pi_\varepsilon} V[\pi_\varepsilon](x) + V[\pi_\varepsilon](x) + \varepsilon \le \gamma \|V_h - V[\pi_\varepsilon]\|_\infty + V[\pi_\varepsilon](x)
+ \varepsilon 
\end{align*}
Combining the above inequality with $\|V_h - V[\pi_\varepsilon]\|_\infty
\le
\frac{\varepsilon}{1-\gamma}$ yields 
$ T^* V_h(x)\le V_h(x) + \frac{\varepsilon}{1-\gamma}. $
Letting $\varepsilon \downarrow 0$ yields
\begin{equation}\label{eq:step3}
   T^* V_h(x) \le V_h(x). 
\end{equation}

\noindent
\textbf{Step 4. Conclusion:}
Combining \eqref{eq:step2} and \eqref{eq:step3}, we obtain
$V_h(x) =T^* V_h(x)$, which is exactly \eqref{dis-be-soft}. This completes the proof.
\end{proof}

\section{Proof of Lemma \ref{dis-operator} and Lemma \ref{dis-opt-control}}\label{proof of lemma dis-operator}

\begin{proof}[Proof of Lemma \ref{dis-operator}]
    Let 
\[
Q(x,u) := r(x,u)h + \gamma\int W(y) \rho_h(y|x,u) dy,
\]
and for any fixed $x=x_0$, \eqref{dis-opt-control-form} can equivalently written as
    \[
   {\pi}_h^*(u)= {\pi}^W_h(u|x_0) = \frac{e^{Q_0(u)/(\lambda h)}}{Z_0}, \quad Q_0(u) := Q(x_0, u), \quad Z_0 := Z(x_0).
    \]
Using the definitions of $r_h^\lambda$, $\rho_h$ and $Q$, and for any $\pi$, direct computation yields
    \[
    \begin{aligned}
        &\E\l[ {r}_h^\lambda(x_0,u;\pi)   + \gamma \E[W(Y^{{\pi}}_{h}) | Y_0 = x_0]\r] \\
        &=\int_{ U} \l[r(x_0,u)h-\lambda h\ln {\pi}(u)  +\gamma\int W(y)\rho_h(y|x_0,u)dy\r]{\pi}(u) du\\
        &=\lambda h\l(\int \frac{Q_0(u)}{\lambda h}{\pi}(u) - {\pi}(u)\ln{\pi}(u)  du\r)= \lambda h\l(\int (\ln {\pi}^*_h(u)+ \ln (Z_0) ){\pi}(u) - {\pi}(u)\ln {\pi}(u)  du\r)\\
        &=\lambda h\l( \ln (Z_0) - \int {\pi}(u)\ln \frac{{\pi}(u)}{{\pi^*_h} (u)} du\r)  = \lambda h\l[ \ln (Z_0) - \text{KL}({\pi}|| {\pi^*_h}) du\r]
    \end{aligned}
    \]
    Since $\text{KL}({\pi}|| {\pi^*_h}) \geq 0$ and the maximum is achieved at ${\pi} = {\pi^*_h}$, therefore, 
    \[
    \begin{aligned}
        T^*W(x) = &\sup_\pi \E\l[ {r}_h^\lambda(x_0,u;\pi)   + \gamma \E[W(Y^{{\pi}}_{h}) | Y_0 = x_0]\r] = \lambda h \ln (Z_0)
    \end{aligned}
    \]
    and the maximum is achieved at ${\pi} = {\pi^*_h}$.
\end{proof}

\begin{proof}[Proof of Lemma \ref{dis-opt-control}]
By Lemma \ref{dis-operator}, one has
\[
\begin{aligned}
    &\ll  T^*W_1(\cdot) - T^*W_2(\cdot) \rl_\infty\\ 
    =&   \ll \lambda h \ln \l[\frac{\int_{U} \exp\l(\frac{r(x,u)h + \gamma\int  W_2(y)+ (W_1(y) - W_2(y)) \rho_h(y|x,u) dy}{\lambda h}\r) du}{\int_{U} \exp\l(\frac{r(x,u)h + \gamma\int  W_2(y) \rho_h(y|x,u) dy}{\lambda h}\r) du}\r]\rl \\
    =&  \ll \lambda h \ln \l[\frac{\int_{U} \exp\l(\frac{r(x,u)h + \gamma\int  W_2(y)\rho_h(y|x,u) dy}{\lambda h}\r) \exp\l(\frac{\gamma\int  (W_1(y) - W_2(y)) \rho_h(y|x,u) dy}{\lambda h}\r) du}{\int_{U} \exp\l(\frac{r(x,u)h + \gamma\int  W_2(y) \rho_h(y|x,u) dy}{\lambda h}\r) du}\r]\rl\\
    \leq& \ll \lambda h \ln \l[\frac{\exp\l(\frac{\gamma \ll W_1 - W_2 \rl_\infty}{\lambda h}\r) \int_{U} \exp\l(\frac{r(x,u)h + \gamma\int  W_2(y)\rho_h(y|x,u) dy}{\lambda h}\r)du }{\int_{U} \exp\l(\frac{r(x,u) + \gamma\int  W_2(y) \rho_h(y|x,u) dy}{\lambda h}\r) du}\r]\rl \\
    =&  \ll \lambda h \ln \l[\exp\l(\frac{\gamma \ll W_1 - W_2 \rl_\infty}{\lambda h}\r) \r]   \rl = \gamma \ll W_1 - W_2 \rl_\infty  .
\end{aligned}
\]
\end{proof}

\bibliographystyle{abbrv}
\bibliography{bib}
	
\end{document}